\documentclass[reqno,a4paper,12pt]{amsart}
\usepackage[margin=3cm]{geometry}
\usepackage[T1]{fontenc}
\usepackage[utf8]{inputenc}
\usepackage{libertinus}
\usepackage{quiver}

\usepackage[stix2, upint, vvarbb]{newtxmath}
\usepackage{amsmath, mathtools, amsthm, thmtools}
\mathtoolsset{mathic}

\usepackage[dvipsnames]{xcolor}
\usepackage[]{hyperref}
\usepackage{cleveref}
\usepackage[final]{microtype}
\usepackage{csquotes}
\usepackage[english]{babel}
\usepackage{tabularx}
\usepackage{booktabs}

\usepackage[backend=biber,style=alphabetic,sorting=nyt]{biblatex}

\hypersetup{
  colorlinks=true,
  linkcolor=Mahogany,
  citecolor=Mahogany,
  urlcolor=Mahogany
}

\tikzcdset{arrow style=math font}

\DeclareFontFamily{U}{min}{}
\DeclareFontShape{U}{min}{m}{n}{<-> udmj30}{}
\newcommand\yo{\!\text{\usefont{U}{min}{m}{n}\symbol{'207}}\!}

\numberwithin{equation}{section}

\declaretheorem[name=Theorem,numberwithin=section]{theorem}
\declaretheorem[name=Lemma,sibling=theorem]{lemma}
\declaretheorem[name=Proposition,sibling=theorem]{proposition}
\declaretheorem[name=Corollary,sibling=theorem]{corollary}
\declaretheorem[name=Porism,sibling=theorem]{porism}

\declaretheorem[numbered=no,name=Proposition]{proposition*}

\declaretheorem[style=definition,sibling=theorem,name=Definition]{definition}
\declaretheorem[style=definition,sibling=theorem,name=Example]{example}
\declaretheorem[style=definition,sibling=theorem,name=Notation]{notation}
\declaretheorem[style=definition,sibling=theorem,name=Construction]{construction}
\declaretheorem[style=definition,sibling=theorem,name=Convention]{convention}

\declaretheorem[style=definition,sibling=theorem,name=Remark]{remark}

\declaretheorem[style=definition,name=Acknowledgements,numbered=no]{acknowledgements}

\makeatletter
\newenvironment{repthm}[1]{%
  \@thm{\th@definition}{}{#1}%
}{%
  \@endtheorem
}
\makeatother

\title{Dold--Kan correspondences via Morita normalization}
\author{Florian Warg}
\address{TU Dresden, Institut für Geometrie, Zellescher Weg 12--14, 01062 Dresden, Germany}
\email{florian.warg@tu-dresden.de}


\newcommand{\id}{\mathrm{id}}
\newcommand{\nd}{\mathrm{nd}}

\newcommand{\bN}{\mathbb{N}}

\newcommand{\bZ}{\mathbb{Z}}

\newcommand{\cat}[1]{\mathbf{#1}}
\newcommand{\Cat}{\cat{Cat}}

\newcommand{\Set}{\cat{Set}}

\newcommand{\Mod}{\cat{Mod}}

\newcommand{\Ch}{\cat{Ch}}

\newcommand{\Du}{\cat{Du}}

\newcommand{\Ring}{\cat{Ring}}

\newcommand{\cA}{\mathscr{A}}
\newcommand{\cB}{\mathscr{B}}
\newcommand{\cC}{\mathscr{C}}
\newcommand{\cD}{\mathscr{D}}

\newcommand{\cP}{\mathscr{P}}

\newcommand{\Nor}{\mathcal{N}}
\newcommand{\N}{\mathrm{N}}

\newcommand{\op}{\mathrm{op}}

\newcommand{\iso}{\cong}

\newcommand{\la}{\langle}
\newcommand{\ra}{\rangle}
\newcommand{\eps}{\varepsilon}

\newcommand*{\noloq}{%
  \nobreak%
  \mathpunct{}%
  \nonscript%
  {:}%
  \mskip6mu
  \relax
}

\DeclareMathOperator{\Aut}{Aut}
\DeclareMathOperator{\End}{End}
\DeclareMathOperator{\im}{im}
\DeclareMathOperator{\coker}{coker}
\DeclareMathOperator{\Rev}{Rev}
\DeclareMathOperator{\Dec}{Dec}
\DeclareMathOperator{\Ob}{Ob}

\DeclareMathOperator{\Fun}{Fun}

\DeclareMathOperator{\sgn}{sgn}

\begin{document}

\begin{abstract}
    The classical Dold--Kan correspondence identifies simplicial modules with connective chain complexes.
    Using Morita theory, we construct a Dold--Kan-type correspondence for every linear category carrying a suitable action of the maximum-preserving simplicial operators.
    The construction functorially assigns a Morita-equivalent normalized category to every such linear category.
    In particular, we obtain Dold--Kan-type correspondences for all crossed simplicial groups.
    We prove a rank formula for the normalized hom-modules and show that a crossed simplicial group is self-dual precisely when its underlying simplicial set of automorphism groups has dimension one.
    Furthermore, we construct normalization-compatible self-dualities and recover the Dwyer--Kan correspondence for duplicial, paracyclic and \(N\)-cyclic modules.
\end{abstract}

\maketitle
\tableofcontents

\section{Introduction}
First proved independently by Dold and Kan in 1958, the Dold--Kan correspondence \cite{dold-58,kan58} gives an equivalence between simplicial \(k\)-modules and connective chain complexes:
\[ \Fun(\Delta^\op, k\Mod) \simeq \Ch_{\geq 0}(k\Mod). \]
It has since inspired numerous generalizations to other indexing categories and to higher categorical settings: 
Pirashvili establishes a Dold--Kan-type theorem for abelian \(\Gamma\)-groups \cite{pirashvilli-2000}, while S{\l}omi\'nska studies such correspondences from the perspective of Morita equivalences of functor categories \cite{slominska}.
Lack and Street construct Dold--Kan correspondences for categories which satisfy certain factorization conditions \cite{lack-street-15}.

Walde develops a homotopy-coherent extension \cite{walde-22} of the Lack--Street framework. Related developments include the Dold--Kan correspondence for stable \(\infty\)-categories by Lurie \cite{lurie07}.
Dyckerhoff categorifies the Dold--Kan correspondence using \(2\)-simplicial stable \(\infty\)-categories \cite{dyckerhoff21}, and Heine extends this approach to the Dwyer--Kan correspondence using \(2\)-duplicial stable \(\infty\)-categories \cite{heine23}.

Returning to the classical setting, Lyubashenko describes the Dold--Kan correspondence using a family of idempotents in the linearized simplex category \cite{lyubashenko-dk-22}.
This family is related by order-reversal to the Dold--Puppe projectors \cite{dold-puppe61}, which are used by Getzler to give a direct proof of the Dwyer--Kan theorem for paracyclic modules \cite{getzler26}. 
In this article, we combine Lyubashenko's and Getzler's insights to develop a functorial framework for constructing Dold--Kan-type correspondences.

\subsection{Main results} Here we briefly recall the necessary definitions and then state the main results of the article.
Let \(\Delta\) be the simplex category. 
We denote by \(\Delta_{\max}, \Delta_{\min} \subseteq \Delta\) the subcategories spanned by the maximum-preserving resp. minimum-preserving simplicial operators. 
These categories are isomorphic and self-dual, so one may be replaced by the other in what follows.
Furthermore, we fix a (nonzero) commutative unital ring \(k\) and refer to \(k\Mod\)-enriched categories as linear categories.
Functors and natural transformations between linear categories are assumed to be \(k\Mod\)-enriched.
If \(\cC\) is an ordinary category, then \(k[\cC]\) denotes its linearization.
Given a small linear category \(\cA\), an \(\cA\)-ring is a linear category \(\cC\) equipped with a functor \(\cA \to \cC\) which is bijective on objects.
\(\cA\)-rings form a full subcategory \(\Ring_{\cA}\) of the coslice category \(\cA \downarrow k\Cat\).

For each \(k[\Delta_{\max}]\)-ring \(\cC\) we construct a linear category \(\Nor(\cC)\).
This gives rise to a functor \(\Nor \colon \Ring_{k[\Delta_{\max}]} \to \Ring_{\Nor(\Delta_{\max})}\) which has the following properties.

\begin{repthm}{\Cref{thm:abstract-dk} \normalfont(Abstract Dold--Kan correspondence)}~
    \begin{enumerate}
        \item The categories \(\cC\) and \(\Nor(\cC)\) are Morita equivalent. In particular, the functor categories \(\Fun_k(\cC^\op, \cA)\) and \(\Fun_k(\Nor(\cC)^\op, \cA)\) are equivalent for every linear, additive and idempotent complete category \(\cA\).
        \item The functor \(\Nor \colon \Ring_{k[\Delta_{\max}]} \to \Ring_{\Nor(k[\Delta_{\max}])}\) is an equivalence.
    \end{enumerate}
\end{repthm}

As an application, we recover the Dwyer–Kan correspondences for duplicial, paracyclic and \(N\)-cyclic modules by computing their normalized indexing categories in \Cref{sec:cyclic-family}.

More generally, \Cref{thm:abstract-dk} yields a Dold--Kan-type correspondence for every crossed simplicial group \(\Delta G\).
The hom-modules of \(\Nor(\Delta G) \coloneq \Nor(k[\Delta G])\) encode information about the underlying simplicial set of automorphism groups
\[ \Delta^\op \to (\Delta G)^\op \xrightarrow{\Delta G(-, 0)} \Set \]
which we denote by \(G\).
In \Cref{thm:rank-thm} we prove a rank formula for the hom-modules based on the nondegenerate simplices of \(G\). We then characterize the crossed simplicial groups \(\Delta G\) whose normalization is self-dual.
\begin{repthm}{\Cref{thm:dim-one-classification}}
    The following are equivalent for a crossed simplicial group \(\Delta G\).
    \begin{enumerate}
        \item There exists an equivalence \((\Delta G)^\op \simeq \Delta G\).
        \item There exists a linear equivalence \(\Nor(\Delta G)^\op \simeq \Nor(\Delta G)\).
        \item The simplicial dimension of \(G\) is one.
    \end{enumerate}
\end{repthm}
\noindent Moreoever, every equivalence \((\Delta G)^\op \simeq \Delta G\) and \(\Nor(\Delta G)^\op \simeq \Nor(\Delta G)\) is necessarily an isomorphism.
This is shown in \Cref{lem:crossed-simplicial-strong-equiv} and \Cref{thm:crossed-self-strong-equiv}.

The proof of \Cref{thm:dim-one-classification} makes use of the following constructions we develop in \Cref{sec:crossed-self-duality}. 
\begin{itemize}
    \item We construct normalization-compatible self-dualities of the paracyclic and paradihedral categories using order adjoints on the integers.
    \item Using a quotient presentation, we extend this construction to all crossed simplicial groups of simplicial dimension one.
\end{itemize}

\begin{acknowledgements}
    The author would like to thank Ulrich Krähmer for many useful discussions and valuable comments on the initial draft of this article, Ivan Bartulović for encouring him to pursue the topic and Tony Zorman for enlightening discussions on promonads.
\end{acknowledgements}

\section{Preliminaries} \label{sec:prelims}
    \subsection{Simplicial constructions} \label{sec:simplicial}

    \begin{notation}
        For every natural number \(n \in \bN\), we define the set \([n] \coloneq \{ 0, 1, \dots, n \}\) and equip it with its natural order.
    \end{notation}

    \begin{definition} We define a category \(\Delta\) with the following data:
        \begin{itemize}
            \item The objects are the natural numbers \(n \in \bN\).
            \item A morphism \(f \in \Delta(m,n)\) is a weakly increasing map \(f\colon [m] \to [n]\).
        \end{itemize}

        The following morphisms of \(\Delta\) will be frequently used.
        \begin{itemize}
            \item \(\delta_i^n \in \Delta(n-1, n)\) is the unique injective morphism that misses \(i \in [n]\).
            \item \(\sigma_j^n \in \Delta(n, n-1)\) is the unique surjective morphism that hits \(j \in [n-1]\) twice.
        \end{itemize}
        We will refer to the \(\delta_i^n\) as \emph{cofaces} and to the \(\sigma^n_j\) as \emph{codegeneracies}.
    \end{definition}

     \begin{definition} \label{def:order-reversal-aut}
        We denote by \(\Rev\colon \Delta \to \Delta\) the order reversal automorphism which is the unique non-trivial auto-equivalence of \(\Delta\). It is the identity on objects and maps \(f \in \Delta(m,n)\) by \(\Rev(f)(i) = n - f(m - i)\) for all \(i \in [m]\).
    \end{definition}

    \begin{remark} \label{rem:rev-action}
        The automorphism \(\Rev\) is involutive and satisfies
        \[ \Rev(\delta^n_i) = \delta^n_{n-i} \quad \text{and} \quad \Rev(\sigma^n_j) = \sigma^n_{n-1-j}. \]
    \end{remark}

    \begin{definition}
        A simplicial object in a category \(\cC\) is a functor \(X\colon \Delta^\op \to \cC\).
    \end{definition}

    \begin{notation}
        Given a simplicial object \(X\colon \Delta^\op \to \cC\), we write \(X_n \coloneq X(n)\). Furthermore, we introduce the notation
        \(d^n_i \coloneq X(\delta^n_i)\colon X_{n} \to X_{n-1}\) and \(s^n_j \coloneq X(\sigma^n_j) \colon X_{n-1} \to X_n\).
    \end{notation}

    \begin{notation}
        The upper indices of \(\delta^n_i\) and \(\sigma^n_j\) may be omitted. The same convention applies to their image under functors.
    \end{notation}

    \begin{remark} \label{rem:delta-generators}
        The following cosimplicial relations in \(\Delta\) will be frequently used:
        \begin{equation} \label{eq:cosimplicial}
            \begin{aligned}
                \delta_i \circ \delta_j &= \delta_{j+1} \circ \delta_i, && \quad i \leq j, \\
                \sigma_j \circ \sigma_i &= \sigma_i \circ \sigma_{j+1}, && \quad  i \leq j, \\
                \sigma_j \circ \delta_i &= \delta_i \circ \sigma_{j-1}, && \quad i < j, \\
                &= \id, &&  \quad  i = j, \ i = j+1 \\
                &= \delta_{i-1} \circ \sigma_j, &&  \quad  i > j + 1.
            \end{aligned}
        \end{equation}
        We note that the \(e_j \coloneq \delta_j \circ \sigma_j \in \Delta(n,n)\) are idempotent. They satisfy
        \[ \sigma_i \circ e_j = \begin{cases}
             e_{j-1} \circ \sigma_i, & \text{if } i < j - 1 \\
             \sigma_{i+1}, & \text{if } i = j - 1 
        \end{cases} \quad \text{and} \quad e_j \circ \delta_i = \begin{cases}
            \delta_i \circ e_j, & \text{if } j < i - 1 \\
            \delta_{i-1}, & \text{if } j = i - 1
        \end{cases}. \]
    \end{remark}

    \begin{definition}
        Let \(X\colon \Delta^\op \to \Set\) be a simplicial set. 
        An element \(x \in X_n\) is \emph{degenerate}, if \(x = s_j(y)\) for some \(y \in X_{n-1}\) and \(j \in [n-1]\). 
        Otherwise \(x\) is \emph{nondegenerate}. We denote the nondegenerate elements by \(X_n^{\nd} \coloneq X_n \setminus \bigcup_{j \in [n-1]}s_j(X_{n-1})\) and define the \emph{dimension} of \(X\) by \(\dim X \coloneq \sup \big( \{ -1\} \cup \{ n \in \bN : X_n^{\nd} \neq \varnothing \} \big)\). 
    \end{definition}

    \begin{proposition} \label{prop:ez-decomp}
        Let \(X\colon \Delta^\op \to \Set\) be a simplicial set.
        Every \(x \in X_n\) factors as
        \[ x = X(\phi)(x') \]
        for a surjective \(\phi \in \Delta(n,r)\) and \(x' \in X^{\nd}_r\). Moreover, this factorization is unique.
    \end{proposition}
    
    \begin{proof}
        Choose \(r \in \bN\) minimal such that we can write \(x = X(\phi)(x')\) for some \(\phi \in \Delta(n,r)\) and \(x' \in X_r\).
        Minimality of \(r\) implies that \(\phi\) is surjective. (Otherwise we could replace \(\phi\) by its image.)
        It also implies that \(x' \in X^{\nd}_r\). For uniqueness we refer to \cite[Proposition 0014]{kerodon}.
    \end{proof}

    \begin{definition} \label{def:ez-decomp}
        We refer to the decomposition \(x = X(\phi)(x')\) in \Cref{prop:ez-decomp} as the \emph{Eilenberg-Zilber decomposition} or \emph{EZ-decomposition} of \(x\) and call \(x'\) the \emph{nondegenerate part} of \(x\).

    \end{definition}

     \begin{remark}
        It will be occasionally useful to regard \(\Delta\) as a full subcategory of \(\Cat\) whose objects are the poset categories \([n]\). In particular, \(\Delta\) has the structure of a strict 2-category.
    \end{remark}

    \begin{proposition} \label{prop:adjoints-in-simplex}
        Regard \(\Delta\) as a 2-category.
        \begin{itemize}
            \item A morphism \(\phi \in \Delta(m,n)\) admits a left adjoint if and only if \(\phi(m) = n\).
            \item A morphism \(\phi \in \Delta(m,n)\) admits a right adjoint if and only if \(\phi(0) = 0\). 
        \end{itemize}
    \end{proposition}

    \begin{proof}
        Let \(\phi \in \Delta(m,n)\) and \(\psi \in \Delta(n,m)\). 
        The adjunction property \(\phi \dashv \psi\) translates to \(\phi(x) \leq y \Leftrightarrow x \leq \psi(y)\) for all \(x \in [m]\) and \(y \in [n]\).
        Let \(S(y) = \{ x \in [m] : \phi(x) \leq y \}\). The adjunction requires \(\psi(y) = \max S(y)\) for all \(y \in [n]\). In particular, \(\phi(0) = 0\). 
        Conversely, \(\phi(0) = 0\) implies that all \(S(y)\) are nonempty, so \(y \mapsto \max S(y)\) is a well-defined weakly increasing map and thus the right adjoint of \(\phi\). 
        The argument for the left adjoint is dual.
    \end{proof}

    We will adopt and extend some of the notation found in \cite[Subsection 04S6]{kerodon}.

    \begin{definition} \label{def:delta-min-max}
        Let \(\Delta_{\min} \subseteq \Delta\) be the wide subcategory consisting of all minimum-preserving morphisms. 
        Similarly, we define \(\Delta_{\max} \subseteq \Delta\) as the wide subcategory consisting of all maximum-preserving morphisms.
    \end{definition}

    \begin{corollary}
        The functor \((\Delta_{\min})^\op \to \Delta_{\max}\) which maps a morphism to its right adjoint is an isomorphism. 
        Its inverse maps a morphism to its left adjoint. 
        Since \(\Rev(\Delta_{\max}) = \Delta_{\min}\), both categories are self-dual.
    \end{corollary}

    \begin{remark} \label{rem:order-adjoint-chain}
        Note that we have the following chain of adjunctions
        \[ \delta^n_n \dashv \sigma^n_{n-1} \dashv \dots \dashv \delta^n_{i+1} \dashv \sigma^n_i \dashv \delta^n_i \dashv \dots \dashv \sigma^n_0 \dashv \delta^n_0. \]
        Every \(\sigma^n_j\) belongs to both \(\Delta_{\min}\) and \(\Delta_{\max}\). 
        The same is true for the \(\delta^n_i\), except at the boundary: \(\delta^n_0\) is missing from \(\Delta_{\min}\) and \(\delta^n_n\) is missing from \(\Delta_{\max}\).
    \end{remark}
    
    \begin{definition}
        The \emph{concatenation functor} \(C_0 \colon \Delta \to \Delta_{\min}\) is the left adjoint of the inclusion \(\Delta_{\min} \hookrightarrow \Delta\). 
        On objects, we have \(C_0(n) = n + 1\) and a morphism \(\phi \in \Delta(m,n)\) is mapped to \(C_0(\phi) \in \Delta_{\min}(m+1, n+1)\) where
        \[ C_0(\phi)(x) = \begin{cases}
            0 & \text{if } x = 0, \\
            \phi(x - 1) + 1 & \text{else}.
        \end{cases} \]
        Similarly, the \emph{concatenation functor} \(C^0\colon \Delta \to \Delta_{\max}\) is the left adjoint of the inclusion \(\Delta_{\max} \hookrightarrow \Delta\). 
        On objects, we have \(C^0(n) = n + 1\) and a morphism \(\phi \in \Delta(m,n)\) is mapped to \(C^0(\phi) \in \Delta_{\max}(m+1, n+1)\) where
        \[ C^0(\phi)(x) = \begin{cases}
            n + 1 & \text{if } x = m + 1, \\
            \phi(x) & \text{else}.
        \end{cases} \]
    \end{definition}

    \begin{remark} \label{rem:decalage-monad}
        The unit of the adjunction \(C_0 \dashv (\iota\colon \Delta_{\min} \hookrightarrow \Delta)\) is the natural transformation \(\delta_{\min} \colon \id_{\Delta} \Rightarrow \iota \circ C_0\) whose components are \((\delta_{0}^{n+1})_{n \in \bN}\). 
        Similarly, the unit of the adjunction \(C^0 \dashv (\iota\colon \Delta_{\max} \hookrightarrow \Delta)\) is the natural transformation \(\delta_{\max} \colon \id_{\Delta} \Rightarrow \iota \circ C^0\) whose components are \((\delta^{n+1}_{n+1})_{n \in \bN}\). 
    \end{remark}

    \begin{definition}
        Let \(X\colon \Delta^\op \to \cC\) be a simplicial object. 
        We define the two variants of the \emph{d\'ecalage} of \(X\) by \(\Dec_0(X) \coloneq X \circ (\iota \circ C_0)^\op\) and \(\Dec^0(X) \coloneq X \circ (\iota \circ C^0)^\op\). 
        The units given in \Cref{rem:decalage-monad} induce natural transformations which we denote by
        \(d_{\min}\colon \Dec_0(X) \Rightarrow X\) and \(d_{\max}\colon \Dec^0(X) \Rightarrow X\). 
    \end{definition}

    \subsection{Linear categories and modules} \label{sec:linear-cats}

    For the remainder of the article we fix a (nonzero) commutative unital ring \(k\). We assume that the reader is familiar with the \(k\Mod\)-enriched analogs of standard categorical constructions (see, e.g.~, \cite[\nopp Sections~1 and 6]{borceux94_2} and \cite{kelly}). 

    \begin{definition}
        A \emph{linear category} \(\cA\) is a category which is \(k\Mod\)-enriched. A \emph{linear functor} \(F\colon \cA \to \cB\) is a \(k\Mod\)-enriched functor. If \(\cA\) is small, we denote by \(\Fun_k(\cA, \cB)\) the linear category of linear functors \(\cA \to \cB\) and linear natural transformations. Unless stated otherwise, functors between linear categories are implicitly assumed to be linear.
    \end{definition}

    If \(X\) is a set, then \(k[X]\) denotes the free \(k\)-module with basis \(X\). We extend this notion to categories.

    \begin{definition}
        The \emph{linearization} of a category \(\cC\) is the linear category \(k[\cC]\) whose objects are the same as those of \(\cC\) and whose hom-modules are the free \(k\)-modules \(k[\cC](x,y) = k[\cC(x,y)]\). Composition is linearly extended from \(\cC\).
    \end{definition}

    Following the perspective introduced by Mitchell \cite{mitchell72}, we regard linear categories as \(k\)-algebras with several objects. 
    From this viewpoint, many standard constructions for rings and modules extend naturally to the many-object setting.

    \begin{definition}
        A \emph{(two-sided) ideal} \(I\) of a linear category \(\cC\) is a family of \(k\)-submodules 
        \[ I(x,y) \subseteq \cC(x,y) \]
        satisfying \(I(y,z) \circ \cC(x,y) \subseteq I(x,z)\) and \(\cC(y,z) \circ I(x,y) \subseteq I(x,z)\) for all \(x,y,z \in \cC\).
        Composition in \(\cC\) descends to the quotient modules.
        We obtain a \emph{quotient category} \(\cC / I\) equipped with a canonical quotient functor \(\cC \to \cC / I\).
    \end{definition}

    \begin{definition}
        The \emph{tensor product} \(\cA \otimes_k \cB\) of linear categories is a linear category consisting of the following data.
        \begin{itemize}
            \item Objects of \(\cA \otimes_k \cB\) are pairs \((a,b)\) with \(a \in \cA\) and \(b \in \cB\).
            \item Hom-modules are \((\cA \otimes_k \cB)((a,b), (a', b')) = \cA(a, a') \otimes_k \cB(b, b')\).
        \end{itemize}
        Composition is defined in the obvious way, using the tensor product on \(k\Mod\). 
    \end{definition}

    \begin{remark}
        The tensor product of linear categories extends to a monoidal structure with tensor unit \(k\), interpreted as a linear category with one object. We will use musical notation to denote mates under the tensor-hom adjunction
        \begin{align*}
            \dfrac{F\colon \cA \otimes_k \cB \to \cC}{F^\sharp\colon \cA \to \Fun_k(\cB, \cC)}, \quad \dfrac{G\colon \cA \to \Fun_k(\cB, \cC)}{G^\flat\colon \cA \otimes_k \cB \to \cC}.
        \end{align*}
    \end{remark}

    \begin{definition}
        Let \(\cA, \cB\) be small linear categories.
        \begin{itemize}
            \item A \emph{right \(\cA\)-module} is a linear functor \(\cA^\op \to k\Mod\).
            \item A \emph{left \(\cB\)-module} is a linear functor \(\cB \to k\Mod\).
            \item An \emph{(\(\cA\),\(\cB\))-bimodule} is a linear functor \(\cA \otimes_k \cB^\op \to k\Mod\).
        \end{itemize}
        The corresponding functor categories are labeled \(\Mod_{\cA}\), \({}_{\cB}\Mod\) and \({}_{\cA}\Mod_{\cB}\). We denote the hom functors associated to an object \(a \in \cA\) by
        \[ \cA_a \coloneq \cA(a, -)\colon \cA \to k\Mod \quad \text{and} \quad \cA^a \coloneq \cA(-,a)\colon \cA^\op \to k\Mod \]
        and identify the category \(\cA\) with the bimodule \(\cA \otimes_k \cA^\op \iso \cA^\op \otimes_k \cA \xrightarrow{\cA(-,-)} k\Mod\).
    \end{definition}

    \begin{definition}
        Given a small linear category \(\cA\) we denote by \(\yo\colon \cA \to \Mod_{\cA}\) the (enriched) Yoneda embedding.
    \end{definition}

    \begin{definition}
        Let \(\cA, \cB, \cC\) be small linear categories. The \emph{tensor product} of bimodules \(M \in {}_{\cA}\Mod_{\cB}\) and \(N \in {}_{\cB}\Mod_{\cC}\) is given by the coend
        \[ M \otimes_{\cB} N \coloneq \int^{\mathrlap{b \in \cB}} M(-, b) \otimes_k N(b, -) \in {}_{\cA}\Mod_{\cC}. \]
    \end{definition}

    We assume familiarity with the calculus of (co)ends and refer to \cite{kelly} and \cite{loregian21} for a detailed treatment of the subject.

    \begin{remark}
        For each small linear category \(\cA\) we have a monoidal category of \(\cA\)-bimodules \(({}_{\cA}\Mod_{\cA}, \otimes_{\cA}, \cA)\).
        More abstractly, this is (up to variance convention) the category of enriched profunctors on \(\cA\) with monoidal structure given by profunctor composition.
    \end{remark}

    \begin{remark}
        Using the natural isomorphisms \(k \otimes_k \cA \iso \cA \iso \cA \otimes_k k\), every left or right module can be interpreted as a bimodule with a trivial \(k\)-action on the other side. Therefore, the coend above also defines the tensor product \(M \otimes_{\cB} N \in k\Mod\) for modules \(M \in \Mod_{\cB}\) and \(N \in {}_{\cB}\Mod\).
    \end{remark}

    We remark that our definition of the tensor product is (up to variance convention) the composition of linear profunctors; see e.g.~\cite[Section~5]{loregian21}.

\section{Generalized Morita theory} \label{sec:morita}
The results of this section are special cases of more general results on modules and profunctors of enriched bicategories (see e.g.~\cite{garner-shulman15}) and are included to keep the article self-contained.
For the remainder of this section, let \(\cA\) and \(\cB\) be small linear categories.

\begin{definition}
    We say that \(\cA\) and \(\cB\) are \emph{Morita equivalent}, if there exists a linear equivalence \(\Mod_{\cA} \simeq \Mod_{\cB}\).
\end{definition}

\begin{remark}
    Our definition recovers the classical notion of Morita equivalence of \(k\)-algebras. We recover a generalized Morita equivalence of rings by taking \(k = \bZ\). Note that every \(\Set\)-enriched equivalence \(\Mod_{\cA} \simeq \Mod_{\cB}\) is automatically \(\bZ\)-linear.
\end{remark}

\subsection{Adjunctions arising from bimodules}
Just as in the classical setting, we may study Morita equivalences by means of bimodules.

\begin{proposition} \label{prop:bimodule-tensor-hom}
    Every bimodule \(X \in {}_{\cA}\Mod_{\cB}\) induces a tensor-hom adjunction
    \[ (-) \otimes_{\cA} X \colon \Mod_{\cA} \leftrightarrows \Mod_{\cB} \noloq \Mod_{\cB}(X,-). \]
\end{proposition}

\begin{proof}
    Let \(M \in \Mod_{\cA}\) and \(N \in \Mod_{\cB}\). We apply the natural isomorphisms
    \begin{align*}
        \Mod_{\cB}(M \otimes_{\cA} X, N) &\iso \textstyle\int_a \int_b k\Mod(M(a) \otimes_k X(a,b), N(b)) \\
        &\iso \textstyle\int_a \int_b k\Mod(M(a), k\Mod(X(a,b), N(b))) \\
        &\iso \Mod_{\cA}(M, \Mod_{\cB}(X, N)). \qedhere
    \end{align*}
\end{proof}

\begin{proposition}
    If \(F\colon \Mod_{\cA} \to \Mod_{\cB}\) is an equivalence, then \(F \iso (-) \otimes_{\cA} X\) where the bimodule \(X\) is given by \((F \circ \yo)^\flat\). 
\end{proposition}

\begin{proof}
    Consider the following natural isomorphisms in \(M \in \Mod_{\cA}\):
    \[
        F(M) \iso F(\textstyle\int^a M(a) \otimes_k \yo(a)) \iso \int^a M(a) \otimes_k F(\yo(a)) \iso M \otimes_{\cA} X. \qedhere
    \]
\end{proof}

We obtain the (co)extension of scalars along a functor as a special case of this adjunction.

\begin{definition}
    A functor \(F\colon \cA \to \cB\) between small linear categories yields a bimodule \(\cB(-, F(-)) \in {}_{\cA}\Mod_{\cB}\). 
    We denote the associated tensor-hom adjunction by
    \[ F_! \colon \Mod_{\cA} \leftrightarrows \Mod_{\cB} \noloq F^* \]
    The other bimodule \(\cB(F(-), -) \in {}_{\cB}\Mod_{\cA}\) induces an adjunction
    \[ F^*\colon \Mod_{\cB} \leftrightarrows \Mod_{\cA} \noloq F_*. \]
    Note that \(F^*\) is naturally isomorphic to restriction along \(F\) by the Yoneda lemma.
\end{definition}

\subsection{Tiny modules and generators}

\begin{remark} \label{rem:necessary-morita}
    If \((-) \otimes_{\cA} X\) is an equivalence, then \(\Mod_{\cB}(X, -)\) is a quasi-inverse by \Cref{prop:bimodule-tensor-hom}. In particular, \(\Mod_{\cB}(X, -)\) is cocontinuous. Since colimits in functor categories are computed pointwise, each functor \(\Mod_{\cB}(X^\sharp(a), -)\) must be cocontinuous as well. This imposes a strong smallness condition on the modules \(X^\sharp(a)\). Additionally, since \(\Mod_{\cB}(X,M) = 0\) holds precisely if \(M = 0\), and the property of being a zero object is verified pointwise, the family \((X^\sharp(a))_{a \in \cA}\) detects zero objects. These observations motivate the following definitions.
\end{remark}

\begin{definition} \label{def:tiny-generator}
    A module \(M \in \Mod_{\cA}\) is called \emph{tiny} if \(\Mod_{\cA}(M,-)\) is cocontinuous. A family of tiny modules \((M_i)_{i \in I} \subseteq \Mod_{\cA}\) is called a \emph{tiny generator} if they jointly reflect zero objects, i.e.~if \(\Mod_{\cA}(M_i, X) = 0\) for all \(i \in I\) implies \(X = 0\).
\end{definition}

\begin{remark}
    More generally, a family of modules \((M_i)_{i \in I} \subseteq \Mod_{\cA}\) is called a generator if the functors \(\Mod_{\cA}(M_i, -)\) are jointly faithful. However, if the \(M_i\) are projective, joint faithfulness is equivalent to jointly reflecting zero objects.
\end{remark}

\begin{proposition} \label{prop:characterization-tiny}
    A module in \(\Mod_{\cA}\) is tiny if and only if it is a retract of a finite direct sum of representables.
\end{proposition}

\begin{proof}
    Observe that all representables are tiny. Indeed, colimits of diagrams in \(\Mod_{\cA}\) are computed pointwise, so the result follows by the Yoneda lemma. 
    Clearly, tiny modules are closed under finite direct sums -- they are also closed under retracts. 
    Suppose \(M \in \Mod_{\cA}\) is a retract of a tiny module \(T \in \Mod_{\cA}\). 
    Then \(\Mod_{\cA}(M, -)\) is a retract of \(\Mod_{\cA}(T,-)\). Since cocontinuity is stable under retracts of functors, we see that \(M\) is also tiny.

    Given a tiny module \(P \in \Mod_{\cA}\), the set of morphisms from all representables to \(P\), i.e. morphisms of the form \(f\colon \cA^a \to P\) with \(a \in \cA\), induces a morphism from the direct sum \(\pi\colon \bigoplus_f \cA^{a} \to P\). We claim that \(\pi\) is an epimorphism. To this end, we will show that \(\coker \pi = 0\). Since all representable \(\cA^a\) are tiny and thus projective, any morphism \(\cA^a \to \coker \pi\) lifts along the cokernel \(P \twoheadrightarrow \coker \pi\) to some \(f\colon \cA^a \to P\). Since every such \(f\) appears in the direct sum, we have \(\Mod_{\cA}(\cA^a, \coker \pi) = 0\) for all \(a \in \cA\) and thus \(\coker \pi = 0\). Now, since \(P\) is projective, \(\pi\) admits a section. Using \(\Mod_{\cA}(P, \bigoplus_{f} \cA^a) \iso \bigoplus_{f} \Mod_{\cA}(P, \cA^a)\), we see that this section is supported only in finitely many direct summands.
\end{proof}

\begin{porism}
    In the proof above we may replace \((\cA^a)_{a \in \cA}\) by a tiny generator \((M_i)_{i \in I}\) to deduce that every tiny module is a retract of a finite direct sum of certain \(M_i\).
\end{porism}

\begin{remark}
    If \(\cA\) corresponds to a classical \(k\)-algebra, then a module is tiny, precisely if it is finitely generated projective. In classical Morita theory, it is common to use the notion of a compact projective generator. A module \(M \in \Mod_{\cA}\) is compact, if \(\Mod_{\cA}(M, -)\) preserves filtered colimits. Using that \(\Mod_{\cA}\) is abelian and cocomplete, it can be shown that \(M\) is compact and projective if and only if \(M\) is tiny.
\end{remark}

\begin{proposition} \label{prop:bimodule-induces-equivalence}
    Let \(X \in {}_{\cA}\Mod_{\cB}\) be a bimodule. The following are equivalent.
    \begin{enumerate}
        \item \((-) \otimes_{\cA} X\colon \Mod_{\cA} \to \Mod_{\cB}\) is an equivalence.
        \item \(X^\sharp \colon \cA \to \Mod_{\cB}\) is fully faithful and \((X^\sharp(a))_{a \in \cA} \subseteq \Mod_{\cB}\) is a tiny generator.
    \end{enumerate}
\end{proposition}

\begin{proof}
    Suppose (1) holds. We have seen in \Cref{rem:necessary-morita} that the family \((X^\sharp(a))_{a \in \cA}\) is a tiny generator. To see that \(X^\sharp\) is fully faithful, we apply the natural isomorphisms
    \begin{align*}
        \cA(a, a') \iso \Mod_{\cA}(\cA^{a}, \cA^{a'}) \iso \Mod_{\cB}(\cA^{a} \otimes_\cA X, \cA^{a'} \otimes_\cA X) \iso \Mod_{\cB}(X^\sharp(a), X^\sharp(a'))
    \end{align*}
    for \(a, a' \in \cA\). Hence (2) holds. Conversely, suppose (2) holds. Then \(\Mod_{\cB}(X, -)\) is cocontinuous. We apply the chain of natural isomorphisms in \(M \in \Mod_{\cA}\)
    \[
        \Mod_{\cB}(X, M \otimes_{\cA} X) \iso \textstyle \int^a M(a) \otimes_k \Mod_{\cB}(X, X^\sharp(a))
        \iso \textstyle \int^a M(a) \otimes_k \cA^a \iso M.
    \]
    This shows that the unit of \((-) \otimes_{\cA} X \dashv \Mod_{\cB}(X, -)\) is an isomorphism. So to conclude (1), it suffices to show that \(\Mod_{\cB}(X,-)\) is conservative. Since it is an exact functor between abelian categories, it suffices to show that it reflects zero objects. However, this is easy to see -- if \(\Mod_{\cB}(X, M) = 0\), then \(\Mod_{\cB}(X^\sharp(a), M) = 0\) for all \(a \in \cA\) and therefore the generator property implies \(M = 0\).
\end{proof}

\begin{corollary} \label{cor:morita-endo}
    Let \(P = (P_i)_{i \in I} \subseteq \Mod_{\cA}\) be a tiny generator. The full subcategory \(\End_{\cA}(P)\) of \(\Mod_{\cA}\) spanned by the objects \(P_i\) is Morita equivalent to \(\cA\). The equivalence is induced by the bimodule \((\End_{\cA}(P) \hookrightarrow \Mod_{\cA})^\flat\).
\end{corollary}

\subsection{Transport of ring structure}
If \(R\) is a (not necessarily commutative) \(k\)-algebra, then an \(R\)-ring is a monoid in the category of \(R\)-bimodules (see e.g.~\cite{bergman74} or \cite[Section~2.1]{boehm08}).
The monoid structures on an \(R\)-bimodule \(A\) are in bijective correspondence with \(k\)-algebra maps \(R \to A\).
This motivates the following definition in the many-object setting.

\begin{definition} \label{def:a-ring}
    An \emph{\(\cA\)-ring} is a functor \(u\colon \cA \to \cC\) which is bijective on objects.
    \(\cA\)-rings form a full subcategory \(\Ring_{\cA}\) of the coslice category \(\cA \downarrow k\Cat\).
    By abuse of notation, we will usually refer to \(\cC\) as the \(\cA\)-ring, assume that \(u\) is the identity on objects, and conflate \(\cA\) with \(u(\cA) \subseteq \cC\).
\end{definition}

\begin{remark}
    If \(\cA\) is a one-object category, then we recover the classical notion of an \(R\)-ring.
    It can be shown that \(\Ring_{\cA}\) is equivalent to the category of monoids in \({}_{\cA}\Mod_{\cA}\) or the category of promonads on \(\cA\) in the bicategory \(k\cat{Prof}\) of small linear categories, profunctors and natural transformations.
\end{remark}

\begin{construction} \label{cons:algebra-transport}
    Given a functor \(\Phi\colon \Mod_{\cA} \to \Mod_{\cB}\) we define the \emph{transport of structure} functor \(\mathcal{T}_\Phi \colon \Ring_{\cB} \to \Ring_{\cA}\).
    Let \(P_a = \Phi(\cA^a) \in \Mod_{\cB}\). From a \(\cB\)-ring \(u\colon \cB \to \cC\), we construct a linear category \(\cC_u\) with the following data:
    \begin{itemize}
        \item \(\Ob(\cC_u) = \Ob(\cA)\), and
        \item \(\cC_u(x, y) = \Mod_{\cC}(u_! (P_x), u_! (P_y))\) with composition of \(\Mod_{\cC}\).
    \end{itemize}
    This category comes with a functor \(u_\Phi\colon \cA \to \cC_u\) which corresponds to \(u_! \circ \Phi \circ \yo_{\cA} \).
    We define \(\mathcal{T}_{\Phi}(\cC) = \cC_u \).
    If \(v\colon \cB \to \cD\) is another \(\cB\)-ring, and \(F\colon \cC \to \cD\) is a morphism of \(\cB\)-rings, then \(F\) induces a functor \(F_\Phi\colon \cC_u \to \cC_v\) via \(\alpha \mapsto \theta \circ F_!(\alpha) \circ \theta^{-1}\) where \(\theta\colon F_! \circ u_! \Rightarrow v_!\) is the natural isomorphism induced by \(F \circ u = v\).
    We set \(\mathcal{T}_\Phi(F) = F_\Phi\).
\end{construction}

The following result is essentially the observation that the Yoneda embedding is dense.

\begin{lemma} \label{lem:yoneda-dense}
    Suppose \(F,G\colon \Mod_{\cA} \to \cC\) preserve colimits and \(\cC\) is cocomplete. Then the following are equivalent.
    \begin{enumerate}
        \item \(F\) and \(G\) are naturally isomorphic.
        \item \(F \circ \yo\) and \(G \circ \yo\) are naturally isomorphic.
    \end{enumerate}
\end{lemma}

\begin{proof}
    It suffices to show that (2) implies (1). The co-Yoneda lemma gives the following isomorphisms that are natural in \(M \in \Mod_{\cA}\): 
    \[
        \textstyle F(M) \iso F(\int^a M(a) \otimes_k \cA^a) \iso \int^a M(a) \otimes_k F(\cA^a) \iso \int^a M(a) \otimes_k G(\cA^a) \iso G(M). \qedhere
    \]
\end{proof}

We will now show that Morita equivalences induce equivalences on ring structures.
In bimodule language, this amounts to observing that conjugation by an invertible bimodule is a strong monoidal equivalence.

\begin{theorem} \label{thm:structure-transport}
    If \(\Phi\colon \Mod_{\cA} \to \Mod_{\cB}\) is an equivalence, then its associated transport of structure functor \(\mathcal{T}_\Phi\colon \Ring_{\cB} \to \Ring_{\cA}\) is an equivalence.
\end{theorem}

\begin{proof}
    Choose an adjoint quasi-inverse \(\Psi\colon \Mod_{\cB} \to \Mod_{\cA}\).
    We will show that \(\mathcal{T}_\Psi\) is a quasi-inverse of \(\mathcal{T}_\Phi\).
    This is mainly a matter of book-keeping.
    Let \(u\colon \cB \to \cC\) be a \(\cB\)-ring.
    In view of \Cref{cons:algebra-transport}, we write \(P_a = \Phi(\cA^a)\) for \(a \in \cA\) and denote by
    \[ J_u\colon \cC_u \to \Mod_{\cC}, \quad a \mapsto u_!(P_a) \]
    the canonical functor which is the identity on morphisms.
    From the adjunction 
    \[ \Mod_{\cC}(u_!(P_a), N) \iso \Mod_{\cB}(P_a, u^*(N)) \]
    we deduce that the family \(u_!(P_a)\) is a tiny generator of \(\Mod_{\cC}\).
    \Cref{prop:bimodule-induces-equivalence} therefore gives an equivalence \(L_u \coloneq (-) \otimes_{\cC_u} J_u^\flat\colon \Mod_{\cC_u} \to \Mod_{\cC}\).

    Now let \(v = u_\Phi\colon \cA \to \cC_u\). 
    We claim that \(L_u \circ v_! \iso u_! \circ \Phi\).
    This is certainly true for their restrictions along the Yoneda embedding:
    \begin{align*}
        L_u(v_!(\cA^a)) \iso L_u(\cC_u^{a}) \iso J_u(a) = u_!(P_a) = u_!(\Phi(\cA^a)).
    \end{align*}
    Hence, by \Cref{lem:yoneda-dense} this extends to \(L_u \circ v_! \iso u_! \circ \Phi\).
    The hom-modules of \((\cC_u)_v\) are formed from all the morphisms between the modules \(v_!(\Psi(\cB^b))\) for \(b \in \cB\).
    Using
    \[ L_u \circ v_! \circ \Psi \circ \yo_{\cB} \iso u_! \circ \Phi \circ \Psi \circ \yo_{\cB} \iso u_! \circ \yo_{\cB} \iso \yo_{\cC} \circ u \]
    and the fact that \(L_u\) is an equivalence, we obtain a fully faithful functor
    \[ \cC \to (\cC_u)_v, \quad u(b) \mapsto b. \]
    Since \(u\) is bijective on objects, this is an isomorphism.
    By construction, it is compatible with the \(\cB\)-ring structure and is natural in \(u\).
    We conclude \(\mathcal{T}_{\Psi} \circ \mathcal{T}_{\Phi} \iso \id\).
    Exchanging the roles of \(\Phi\) and \(\Psi\) proves the claim.
\end{proof}

\subsection{Change of coefficients}
Our notion of Morita equivalence relies on categories of the form \(\Fun_k(\cA^\op, k\Mod)\).
However, the equivalence extends to functor categories where \(k\Mod\) is replaced by a suitable additive and idempotent complete category.

\begin{definition}
    Let \(\cA\) be a small linear category. We define \(\widehat{\cA}\) to be its Cauchy-completion. We will use the following explicit construction. For every idempotent morphism \(e\) on a finite direct sum of representables choose an image object \(\im e \in \Mod_{\cA}\). Then \(\widehat{\cA}\) is the full subcategory of \(\Mod_{\cA}\) spanned by the objects \(\im e\) for every \(e\).
\end{definition}

\begin{proposition}
    Let \(\cC\) be a \(k\)-linear category which is additive and idempotent complete. Restriction along the Yoneda embedding \(\cA \hookrightarrow \widehat \cA\) yields a \(k\)-linear equivalence \(\Fun_k(\widehat{\cA}^\op, \cC) \simeq \Fun_k(\cA^\op, \cC)\).
\end{proposition}

\begin{proof}

    This amounts to observing that a \(k\)-linear functor extends over finite direct sums since \(\cC\) is additive and extends over retracts since idempotents in \(\cC\) split.
    We refer to \cite{BorceuxDejean86} for more details.
\end{proof}

\begin{corollary} \label{cor:morita-coefficients}
    Suppose \(\cA\) and \(\cB\) are Morita equivalent. For every \(k\)-linear additive and idempotent complete category \(\cC\) we have an equivalence \(\Fun_k(\cA^\op, \cC) \simeq \Fun_k(\cB^\op, \cC)\). 
\end{corollary}

\begin{proof}
    Let \(F\colon \Mod_{\cA} \to \Mod_{\cB}\) be an equivalence. Clearly, \(F\) induces an equivalence \(\widehat \cA \simeq \widehat \cB\). We conclude the desired result from the chain of equivalences
    \[ \Fun_k(\cA^\op, \cC) \simeq \Fun_k(\widehat \cA^\op, \cC) \simeq \Fun_k(\widehat \cB^\op, \cC) \simeq \Fun_k(\cB^\op, \cC). \qedhere \]
\end{proof}

\section{Morita normalization} \label{sec:morita-normalization}

\subsection{The Dold--Puppe projectors} The basis of our normalization approach is a family of idempotent projectors which can be constructed from the counits of the lower and upper d\'ecalage comonads.

\begin{construction}
    Let \(X\colon \Delta^\op \to k\Mod\) be a simplicial module. We introduce the notation \(K(X) = \ker (d_{\min}\colon \Dec_0(X) \Rightarrow X)\). For each \(n\), we have an exact sequence
\[\begin{tikzcd}[ampersand replacement=\&,cramped]
	0 \& {K(X)_n} \&\& {\Dec_0(X)_n} \&\& {X_n} \& 0
	\arrow[from=1-1, to=1-2]
	\arrow[hook, from=1-2, to=1-4]
	\arrow["{(d_{\min})_n}", from=1-4, to=1-6]
	\arrow[from=1-6, to=1-7]
\end{tikzcd}\]
which is right split by \(s_0\). 
This yields a retraction \(X_{n+1} = \Dec_0(X)_n \to K(X)_n\) which acts via \(\id - s_0 \circ d_0\). 
By iterating this construction, we obtain a chain of retractions
\[ K^0(X)_n \leftrightarrows K^1(X)_{n-1} \leftrightarrows K^2(X)_{n-2} \leftrightarrows \dots \leftrightarrows K^{n-1}(X)_1 \leftrightarrows K^n(X)_0 \]
which have the same form. Since \(\Dec_0\), and thus, \(K\) shifts the index of cofaces and codegeneracies, the resulting idempotent on \(K^0(X)_n = X_n\) is given by
\[ (\id - s_{n-1} \circ d_{n-1}) \circ \dots \circ (\id - s_1 \circ d_1) \circ (\id - s_0 \circ d_0). \]
The analogous construction with \(\Dec^0(X)\) and \(\ker(d_{\max}\colon \Dec^0(X) \Rightarrow X)\) yields another idempotent which can be identified as
\[ 
    (\id - s_0 \circ d_1) \circ (\id - s_1 \circ d_2) \circ \dots \circ (\id - s_{n-1} \circ d_n).
\]

\end{construction}
Under the Yoneda lemma, these projectors correspond to certain morphisms in \(k[\Delta]\) which we introduce below.

\begin{definition}
    For \(n \in \bN\) we define the projectors
    \begin{align*}
        p_n &\coloneq (\id - \delta_0 \circ \sigma_0) \circ (\id - \delta_1 \circ \sigma_1) \circ \dots \circ (\id - \delta_{n-1} \circ \sigma_{n-1}) \in k[\Delta_{\max}] \\
        q_n &\coloneq (\id - \delta_n \circ \sigma_{n-1}) \circ (\id - \delta_{n-1} \circ \sigma_{n-2}) \circ \dots \circ (\id - \delta_1 \circ \sigma_0) \in k[\Delta_{\min}].
    \end{align*}
    By convention, \(p_0\) and \(q_0\) are defined as \(\id_0\), corresponding to the empty composition.
\end{definition}

\begin{remark}
    The construction of \(X(q_n)\) already appears in one of the earliest publications on the Dold--Kan correspondence by Dold and Puppe \cite[\nopp Section~3.9]{dold-puppe61}.
    The truncated loop space construction which we denoted by \(K(X)\) is attributed to J. Moore. 
    It can also be found in \cite[\nopp Exercise~8.3.9]{weibel94} under the name ``brutal loop space''.
\end{remark}

\begin{remark}
    The projector families are related via \(q_n = \Rev(p_n)\) (cf.~\Cref{def:order-reversal-aut}).
\end{remark}

\begin{remark} \label{rem:projector-expansion}
    We will use the following elementary observations about the projectors. If we write \(p_n = \prod_{i=0}^{n-1} (\id - \delta_i \circ \sigma_i)\), then we may expand it in the form
    \[ 
        (\id - \delta_0 \circ \sigma_0) \circ \prod_{i=1}^{n-1} (\id - \delta_i \circ \sigma_i) = \delta_0 \circ a_0 + \prod_{i=1}^{n-1} (\id - \delta_i \circ \sigma_i) = \dots = \id + \sum_{i=0}^{n-1} \delta_i \circ a_i
    \]
    where \(a_i \in k[\Delta_{\max}](n, n-1)\). 
    Analogously, we can expand it to the form
    \[
        \prod_{i=0}^{n-2} (\id - \delta_{i} \circ \sigma_{i}) \circ (\id - \delta_{n-1} \circ \sigma_{n-1}) = b_{n-1} \circ \sigma_{n-1} + \prod_{i=0}^{n-2} (\id - \delta_i \circ \sigma_i) = \dots = \id + \sum_{i=0}^{n-1} b_i \circ \sigma_i
    \]
    where \(b_i \in k[\Delta_{\max}](n-1, n)\). By applying \(\Rev\), we obtain an analogous formula
    \[ q_n = \id + \sum_{i = 1}^n \delta_i \circ a_i' = \id + \sum_{i=0}^{n-1} b_i' \circ \sigma_i \]
    where \(a_i' \in k[\Delta_{\min}]\) and \(b_i' \in k[\Delta_{\min}]\). Note that we have shifted bounds for the cofaces.
\end{remark}
    
\subsection{Dold--Kan correspondences} \label{sub:dk}
    Throughout this section, we will fix an object 
    \[ (u\colon k[\Delta_{\max}] \to \cC) \in \Ring_{k[\Delta_{\max}]}. \]
    All results carry over to \(\Ring_{k[\Delta_{\min}]}\) via the isomorphism
    \[ (-) \circ k[\Rev]\colon \Ring_{k[\Delta_{\min}]} \to \Ring_{k[\Delta_{\max}]}. \]
    Under this identification, we replace \(p_n\) by \(\Rev(p_n) = q_n\).

    \begin{definition}
        Let \(X \in \Mod_{\cC}\). The \emph{degenerate submodule} in degree \(n \in \bN\) is defined as the \(k\)-module \(D_n(X) \coloneq \sum_{j < n} \im s^n_j \subseteq X_n\). We extend this definition to a family of functors \((D_n\colon \Mod_{\cC} \to k\Mod)_{n \in \bN}\) in the obvious way.
    \end{definition}

    \begin{proposition} \label{prop:projector-idempotent}
        The morphisms \(\id - p_n\) are idempotent and exhibit \(D_n\) as a retract of the evaluation functor \((-)_n\).
    \end{proposition}

    \begin{proof}
        By \Cref{rem:projector-expansion} we may write \(p_n - \id = \sum_{i < n} b_i \circ \sigma_i\). Using \Cref{rem:delta-generators}, we see that \(\sigma_i \circ p_n = 0\) for \(i < n\) which then implies \((p_n - \id) \circ p_n = 0\). Combining these observations, we conclude \(D_n(X) = \ker X(p_n)\) for all \(X \in \Mod_{\cC}\).
    \end{proof}

    \begin{proposition} \label{prop:projector-decomp}
        The morphism \(p_n\) induces \(\cC^n \iso P_n \oplus Q_n\) in \(\Mod_{\cC}\) where
        \[ P_n \coloneq \im \yo(p_n) = \bigcap_{i = 0}^{n-1} \ker \yo(\sigma^n_i) \quad \text{and} \quad Q_n \coloneq \ker \yo(p_n) = \sum_{i = 0}^{n-1} \im \yo(\delta^n_i). \]
    \end{proposition}
    
    \begin{proof}
        The proof of \Cref{prop:projector-idempotent} already shows \(\im \yo(p_n) \subseteq \bigcap_{i < n} \ker \yo(\sigma_i) \). Conversely, if \(\sigma_i \circ x = 0\) for all \(i < n\), then \(p_n \circ x = x\), so \(\im \yo(p_n) = \bigcap_{i < n} \ker \yo(\sigma_i) \) as claimed. 
        The characterization of \(\ker p_n\) is dual. By \Cref{rem:projector-expansion} we may write \(p_n - \id = \sum_{i < n} \delta_i \circ a_i\) and using \Cref{rem:delta-generators} we deduce \(p_n \circ \delta_i = 0\) for \(i < n\). Since \(\ker \yo(p_n) = \im \yo(p_n - \id)\), the result follows.
    \end{proof}

    \begin{lemma} \label{lem:dold-kan-tiny-generator}
        The family \((P_n)_{n \in \bN}\) is a tiny generator of \(\Mod_{\cC}\).
    \end{lemma}

    \begin{proof}
        Each \(P_n\) is a retract of a representable and thus tiny (cf.~\Cref{prop:characterization-tiny}). For \(X \in \Mod_{\cC}\) we have \(\Mod_{\cC}(P_n, X) \iso X_n / D_n(X)\). If \(X \neq 0\), let \(n \in \bN\) be minimal such that \(X_n \neq 0\). Then \(D_n(X) = 0\) and thus \(\Mod_{\cC}(P_n, X) \neq 0\) which shows that \((P_n)\) is a generator.
    \end{proof}

    \begin{definition}
        The \emph{normalization} of \(\cC\) is the linear category \(\Nor(\cC)\) where
        \begin{itemize}
            \item \(\Ob \Nor(\cC) = \bN\), and
            \item \(\Nor(\cC)(m, n) = \Mod_{\cC}(P_{u(m)}, P_{u(n)})\) with composition as in \(\Mod_{\cC}\).
        \end{itemize}
        If \(\cC = k[\cD]\) is the linearization of a small category \(\cD\), we write \(\Nor(\cD) \coloneq \Nor(\cC)\).
        In the case \(u = \id_{k[\Delta_{\max}]}\), we obtain a canonical embedding \(J\colon \Nor(\Delta_{\max}) \hookrightarrow \Mod_{k[\Delta_{\max}]}\).
        We denote the transport of structure functor associated to \(\Phi = (-) \otimes_{\Nor(\Delta_{\max})} J^\flat\) by
        \[ \Nor \coloneq \mathcal{T}_{\Phi} \colon \Ring_{k[\Delta_{\max}]} \to \Ring_{\Nor(\Delta_{\max})}. \]
        Extension of scalars sends the image of \(p_n\) to the image of \(u(p_n)\).
        This identifies the transported category with \(\Nor(\cC)\) as defined above.
    \end{definition}

    \begin{theorem}[Abstract Dold--Kan correspondence] \label{thm:abstract-dk}~
        \begin{enumerate}
            \item The categories \(\cC\) and \(\Nor(\cC)\) are Morita equivalent.
            In particular, the functor categories \(\Fun_k(\cC^\op, \cA)\) and \(\Fun_k(\Nor(\cC)^\op, \cA)\) are equivalent for every linear, additive and idempotent complete category \(\cA\).
            \item The functor \(\Nor \colon \Ring_{k[\Delta_{\max}]} \to \Ring_{\Nor(\Delta_{\max})}\) is an equivalence.
        \end{enumerate}
    \end{theorem}
    
    \begin{proof}
        (1) follows from \Cref{prop:bimodule-induces-equivalence} and \Cref{lem:dold-kan-tiny-generator} applied to the bimodule 
        \[ (\Nor(\cC) \to \Mod_{\cC})^\flat \in {}_{\Nor(\cC)} \Mod_{\cC}, \]
        and then applying \Cref{cor:morita-coefficients}. (2) follows by \Cref{thm:structure-transport}.
    \end{proof}

    \begin{construction}
        We give a description of \(\Nor(\cC)\) via quotients.
        The Yoneda embedding \(\cC \to \Mod_{\cC}\) gives a natural isomorphism of \(k\)-modules
        \[ p_n \circ \cC(m,n) \circ p_m \xrightarrow{ \ \sim \ } \Nor(\cC)(m,n). \]
        Denote by \(L \coloneq p_n \circ (-)\) and \(R \coloneq (-) \circ p_m\) the endomorphisms of \(\cC(m,n)\) that compose with the projector. \Cref{prop:projector-decomp} shows \(\ker L = \sum_{i=0}^{n-1} \delta_i \circ \cC(m, n-1)\). 
        The proof of \Cref{prop:projector-idempotent} shows \(\ker R = \sum_{j=0}^{m-1} \cC(m-1,n) \circ \sigma_j\).
        Since \(R \circ L = L \circ R\), we conclude
        \[ 
            \Nor(\cC)(m,n) \iso \cC(m,n) / (\ker L + \ker R).
        \]
        Under this identification, composition of the classes of \(f \in \cC(m,r)\) and \(g \in \cC(r,n)\) is given by \([g] \circ [f] = [g \circ p_r \circ f]\).

        Even more is true if \(u\colon k[\Delta_{\max}] \to \cC\) is the linearization of a functor \(\Delta_{\max} \to \cD\).
        In this case, \(\ker L\) and \(\ker R\) are spanned by subsets of \(\cD(m,n)\),
        and thus \(\Nor(\cD)(m,n)\) is a free module.
        A basis of the quotient is given by the non-vanishing classes of \(f \in \cD(m,n)\).
        Note that \([f] = 0\) holds precisely if
        \begin{itemize}
            \item \(f = \delta_i \circ f'\) for some \(i < n\), or
            \item \(f = f' \circ \sigma_j\) for some \(j < m\).
        \end{itemize}

    \end{construction}

    \begin{construction} \label{construction:normalized-module}
        We give a description of the equivalence \(\Mod_{\cC} \to \Mod_{\Nor(\cC)}\).
        The right-adjoint of the tensor-hom adjunction associated to the bimodule \((\Nor(\cC) \to \Mod_{\cC})^\flat\) maps \(X \in \Mod_{\cC}\) to a module \(\N(X) \in \Mod_{\Nor(\cC)}\) with components
        \[ \N(X)_n = \Mod_{\cC}(P_n, X) \iso X_n / D_n(X). \]
        Under this identification, the right action of \(\Nor(\cC)\) is given as follows. For \(x \in X_n\) and \(f \in \cC(m,n)\) we have \([x] \triangleleft [f] = [x \triangleleft (p_n \circ f)]\).
    \end{construction}

    \begin{example}[The trivial case] \label{ex:trivial}
        We consider \(u = \id_{k[\Delta_{\max}]}\).
        Given a morphism \(\phi \in \Delta_{\max}(m,n)\), we have \([\phi] \neq 0\) precisely if \([n-1] \subseteq \phi([m])\) and \(\phi\) is injective.
        Only the identity morphisms satisfy this property.
        It follows that \(\Nor(\Delta_{\max})\) is the linearization of a discrete category and thus \(\Mod_{\Nor(\Delta_{\max})}\) is simply the category of \(\bN\)-graded \(k\)-modules. 
    \end{example}

    \begin{example}[The classical case] \label{ex:classical-dk}
    We consider the inclusion \(u = (k[\Delta_{\max}] \hookrightarrow k[\Delta]) \). As in the previous example, given \(\phi \in \Delta(m,n)\) we have \([\phi] \neq 0\) precisely if \([n-1] \subseteq \phi([m])\) and \(\phi\) is injective. We conclude
    \[ \Nor(\Delta)(m,n) \iso \begin{cases}
        \la [\id] \ra & \text{if } m = n, \\
        \la [\delta_n] \ra & \text{if } m = n - 1, \\
        0 & \text{otherwise.}
    \end{cases} \]
    In view of \Cref{construction:normalized-module}, we compute \(p_n \circ \delta_n\). 
    The cosimplicial relations (\ref{eq:cosimplicial}) yield \(p_n \circ \delta_n = \big( \sum_{i=0}^n (-1)^{i} \delta_{n-i} \big) \circ p_{n-1}\). 
    If \(\delta_n\) is replaced by \((-1)^n \delta_n\), we obtain the familiar action \([x] \triangleleft [(-1)^n \delta_n] = [x \triangleleft \sum_{i=0}^n (-1)^i \delta_i]\).
    Note that \(\Nor(\Delta)\) is isomorphic to \(\Ch^{\geq 0}\); the quotient of the free linear category on the quiver \((\partial^n \colon n \to n+1)_{n \in \bN}\) modulo the relations \(\partial^{n+1} \circ \partial^n = 0\). 
    The right modules of \(\Ch^{\geq 0}\) are precisely the nonnegative chain complexes in \(k\Mod\). 
    We recover the classical Dold--Kan correspondence 
    \[ \Fun(\Delta^\op, k\Mod) \iso \Mod_{k[\Delta]} \simeq \Mod_{\Ch^{\geq 0}} \iso \Ch_{\geq 0}(k\Mod). \]
    \end{example}

\subsection{Interaction with standard constructions} \label{sec:normalization-standard-constructions}
For this section we fix a functor \(u\colon k[\Delta_{\max}] \to \cC\) which is bijective on objects.
For convenience we will assume, without loss of generality, that \(u\) is the identity on objects.

\begin{remark}[Base change and duality] \label{rem:normalization-base-change}
    Suppose we have a commutative diagram in \(k\Cat\) of the form
\[\begin{tikzcd}[ampersand replacement=\&,cramped]
	\cC \&\& \cD \\
	\\
	{k[\Delta_{\max}]} \&\& \cA
	\arrow["F", from=1-1, to=1-3]
	\arrow["u", from=3-1, to=1-1]
	\arrow["\Psi", from=3-1, to=3-3]
	\arrow["v"', from=3-3, to=1-3]
\end{tikzcd}\]
    where all functors are bijective on objects.
    Then \(v \circ \Psi\) is a \(k[\Delta_{\max}]\)-ring structure, so normalization yields a functor
    \(\Nor(F)\colon \Nor(\cC) \to \Nor(\cD)\).
    In particular, a diagram of the form
    \[\begin{tikzcd}[ampersand replacement=\&,cramped]
        {\cC^\op} \&\& \cC \\
        \\
        {k[\Delta_{\max}]^\op} \&\& {k[\Delta_{\max}]}
        \arrow["F", from=1-1, to=1-3]
        \arrow["{u^\op}", from=3-1, to=1-1]
        \arrow["\Psi", from=3-1, to=3-3]
        \arrow["u"', from=3-3, to=1-3]
    \end{tikzcd}\]
    where \(F\) and \(\Psi\) are isomorphisms, induces a self-duality \(\Nor(\cC)^\op \iso \Nor(\cC)\).
\end{remark}

\begin{remark}[Quotients]
    Let \(I \subseteq \cC\) be an ideal and set
    \[ \widebar{I}(m,n) = u(p_n) \circ I(m,n) \circ u(p_m). \] 
    Then \(\widebar{I} \subseteq \Nor(\cC) \) is an ideal.
    The composite \(k[\Delta_{\max}] \to \cC \to \cC / I\)
    induces a universal functor \(\Nor(\cC) / \widebar{I} \to \Nor(\cC / I)\) which is an isomorphism (and bijective on objects).
\end{remark}

\begin{remark}[Localization]
    Let \(\eta \colon \id_{\cC} \Rightarrow \id_{\cC}\) be a natural transformation.
    We define \(\widebar{\eta}\colon \id_{\Nor(\cC)} \Rightarrow \id_{\Nor(\cC)}\) via \(\widebar{\eta}_n = u(p_n) \circ \eta_n \circ u(p_n)\).
    The composite \(k[\Delta_{\max}] \to \cC \to \cC[\eta^{-1}]\)
    induces a universal functor \(\Nor(\cC)[\widebar{\eta}^{-1}] \to \Nor(\cC[\eta^{-1}])\) which is an isomorphism.
\end{remark}

\section{Crossed simplicial groups} \label{sec:crossed-groups}
Crossed simplicial groups furnish \(\Delta\) with a sequence of automorphism groups in each degree.
They generalize Connes' cyclic category \(\Lambda\) which is introduced in \Cref{def:cyclic-category}.
We refer the reader to \cite{krasauskas87,fiedorowicz-loday-91,dyckerhoff-kapranov-15} for more details on the subject.

\subsection{Preliminaries}
\begin{definition}
    A \emph{crossed simplicial group} is a functor \(i\colon \Delta \to \Delta G\) which is bijective on objects and satisfies the property that every morphism \(f \in \Delta G(i(m), i(n))\) factors uniquely as \(f = i(\phi) \circ g\) for \(\phi \in \Delta(m,n)\) and \(g \in \Aut_{\Delta G}(i(m))\). We call \(\phi \circ g\) the \emph{normal form} of \(f\).
    Crossed simplicial groups form a full subcategory of the coslice category \(\Delta \downarrow \Cat\).
\end{definition}

\begin{convention}
    By abuse of notation, we will usually assume that \(i\) is the identity on objects and conflate \(\Delta\) with \(i(\Delta) \subseteq \Delta G\).
\end{convention}

\begin{definition}
    Given a crossed simplicial group \(\Delta G\), we denote by \(G\) the simplicial set
    \[ \Delta^\op \to (\Delta G)^\op \xrightarrow{\Delta G(-, 0)} \Set. \]
    The normal form of \(\Delta G\) allows us to naturally identify \(G_n = \Delta G(n,0) \iso \Aut_{\Delta G}(n)\).
    Under this identification, we will write \(h = g \triangleleft \phi\) whenever \(\omega_{n} \circ h = \omega_{m} \circ g \circ \phi\) and \(\omega_{i} \in \Delta(i,0)\) is the terminal morphism.
\end{definition}

\begin{definition}
    Let \(f\colon \Delta G \to \Delta H\) be a morphism of crossed simplicial groups with structures given by \(i\colon \Delta \to \Delta G\) and \(j\colon \Delta \to \Delta H\).
    We define the kernel of \(f\) as the pullback
    \[\begin{tikzcd}[ampersand replacement=\&,cramped]
        {\Delta K} \&\& \Delta \\
        \\
        {\Delta G} \&\& {\Delta H}
        \arrow[from=1-1, to=1-3]
        \arrow[from=1-1, to=3-1]
        \arrow["j"', from=1-3, to=3-3]
        \arrow["f", from=3-1, to=3-3]
    \end{tikzcd}\]
    equipped with the universal functor \(\Delta \to \Delta K\) induced by \(\Delta \xleftarrow{\id} \Delta \xrightarrow{i} \Delta G\).
    This makes \(\Delta K \to \Delta G\) a morphism of crossed simplicial groups.
\end{definition}

\begin{definition}
    A morphism of crossed simplicial groups \(p \colon \Delta G \to \Delta H\) is called an \emph{extension}, if it is full.
    Let \(\Delta K\) be the kernel of \(p\).
    Then 
    \[ \Delta K \to \Delta G \xrightarrow{p} \Delta H \]
    induces a degreewise short exact sequence of groups
    \[ 1 \to K_n \to G_n \to H_n \to 1. \]
\end{definition}

\begin{theorem} \label{thm:classification-extensions}
    For every crossed simplicial group \(\Delta G\) there exists an extension \(\Delta G \to \Delta H\)
    with kernel \(\Delta K\) such that \(K\) is a simplicial group and \(\Delta H\) is a crossed simplicial group of fundamental type listed in \cite[Table 2]{dyckerhoff-kapranov-15}.
    This extension is unique up to isomorphism.
\end{theorem}

\begin{proof}
    \cite[\nopp Theorem~I.7]{dyckerhoff-kapranov-15}.
\end{proof}

\begin{proposition} \label{prop:crossed-reverse-codeg}
    If \(g \circ \sigma_i = h \circ \sigma_j\) then \(i = j\) and \(g = h\) for all automorphisms \(g,h\).
\end{proposition}

\begin{proof}
    From the cosimplicial identities \(\sigma_i \circ \delta_i = \id = \sigma_i \circ \delta_{i+1}\) we deduce
    \[ g \circ \sigma_i = h \circ \sigma_j \implies h \circ \sigma_j \circ \delta_i = h \circ \sigma_j \circ \delta_{i+1} \implies \sigma_j \circ \delta_{i} = \sigma_j \circ \delta_{i+1}. \]
    Evaluating these maps on \(i\) yields \(\sigma_j(i+1) = \sigma_j(i)\) and thus \(i = j\) as well as \(g = h\).
\end{proof}

\begin{definition}
    For \(g \in \Aut_{\Delta G}(n)\) we call \(j \in [n-1]\) a \emph{degeneracy direction} of \(g\), if \(\sigma_j \circ g\) is degenerate in the simplicial set \(\Delta G(-, n - 1)|_{\Delta^\op}\).
    We denote the set of degeneracy directions of \(g\) by \(D(g) \subseteq [n-1]\).
\end{definition}

\begin{proposition}
    Suppose \(g \in \Aut_{\Delta G}(n)\) has EZ-decomposition \(g = h \triangleleft \phi\) in \(G\) with \(h \in G^{\nd}_r\).
    Then \(g\) has exactly \(n - r\) degeneracy directions.
\end{proposition}

\begin{proof}
    Let \(S \subseteq [n-1]\) be the set of indices \(j\) such that \(\phi(j) = \phi(j+1)\).
    Note that \(|S| = n - r\), since \(\phi\) is surjective.
    For \(j \in S\) we have \(\phi = \phi_j \circ \sigma_j\) for a unique \(\phi_j \in \Delta(n-1,r)\).
    We set \(h_j = h \triangleleft \phi_j\). It follows that \(h_j \circ \sigma_j = \sigma_{\pi(j)} \circ g\) for a unique index \(\pi(j) \in [n-1]\).
    By \Cref{{prop:crossed-reverse-codeg}}, the assignment \(j \mapsto \pi(j)\) is an injective map \(S \to D(g)\).
    It is also surjective.
    Indeed, given \(k \in D(g)\) write \(\sigma_k \circ g = f \circ \sigma_j\).
    It follows that \(f\) is a split epic endomorphism and thus an automorphism.
    Hence, \(g = f \triangleleft \sigma_j\). Uniqueness of the EZ-decomposition yields \(j \in S\) and thus \(\pi(j) = k\).
\end{proof}

\subsection{Structure of the normalized category}

\begin{theorem} \label{thm:rank-thm}
    Let \(g \in G_m\) have its nondegenerate part in degree \(r\). The number of morphisms \(\phi \in \Delta(m,n)\) such that \([\phi \circ g] \neq 0\) in \(\Nor(\Delta G)\), is given by the binomial coefficient
    \[ \binom{r+1}{m-n+1} \]
    which we define to be zero when the lower index is outside the usual range.
    In particular, \(\Nor(\Delta G)(m,n)\) is a free \(k\)-module of rank
    \[ \sum_{r=0}^m |G_r^{\nd}| \cdot \binom{m}{r} \cdot \binom{r+1}{m-n+1}. \]
\end{theorem}

\begin{proof}
    Let \(\phi \circ g \in \Delta G(m,n)\) be in normal form.
    We have \([\phi \circ g] \neq 0\) in \(\Nor(\Delta G)\), precisely if \([n-1] \subseteq \im \phi\) and \(\phi(j) < \phi(j + 1)\) for all \(j \in D(g)\).
    Suppose that \(g\) has its nondegenerate part in degree \(r\) and put \(s = |D(g)| = m - r\).
    The image condition on \(\phi\) says that it is either
    \begin{enumerate}
        \item surjective onto \([n]\), with exactly \(n\) strict increases, or
        \item surjective onto \([n-1]\), with exactly \(n-1\) strict increases.
    \end{enumerate}
    Since \(\phi(0) = 0\), the positions of the increases determine \(\phi\) completely.
    The \(s\) positions in \(D(g)\) are fixed, the other positions can be chosen freely from the complementary set, which has \(r\) elements. Therefore, the count is
    \[ \binom{r}{n - s} + \binom{r}{n - 1 - s} = \binom{r + 1}{n - s} = \binom{r + 1}{m - n + 1}. \]
    Here, the first summand comes from case (1) and the second summand from case (2).
    If \(n = 0\), case (2) does not occur, and the same formula holds. By uniqueness of the EZ-decomposition, each \(h \in G^{\nd}_r\) gives exactly one element \(g \in G_m\) for each surjective \(\psi \in \Delta(m, r)\).
    Since there are \(\binom{m}{r}\) surjections in \(\Delta(m,r)\), the rank formula follows.
\end{proof}

\begin{corollary} \label{cor:crossed-normalization-nonzero}
    Let \(d_m = \max \{ r \leq m : G_r^{\nd} \neq \varnothing \}\). Then
    \[
        \Nor(\Delta G)(m,n) \neq 0 \iff m - d_m \leq n \leq m + 1.
    \]
\end{corollary}
\begin{proof}
    Note that \(r \leq m\) contributes in the rank formula precisely if \(G^{\nd}_r \neq \varnothing\) and the binomial coefficients are nonzero, i.e.~\(0 \leq m - n + 1 \leq r + 1\).
\end{proof}

\begin{corollary}
    The rank formula and EZ-decomposition imply
    \begin{enumerate}
        \item \(\Nor(\Delta G)(m,0) \iso k[G_m^{\nd}]\),
        \item \(\Nor(\Delta G)(m, m+1) \iso k[G_m]\),
        \item If \(G_r^{\nd} \neq \varnothing\), then \(\Nor(\Delta G)(n + r, n) \neq 0\) for all \(n \in \bN\).
    \end{enumerate}
    In particular, all \(P_m\) are nonzero due to \(\Nor(\Delta G)(m, m+1) \iso \Mod_{k[\Delta G]}(P_m, P_{m+1})\).
\end{corollary}

An interesting consequence of this observation is that modules of a crossed simplicial group cannot be equivalent to modules of a classical ring.

\begin{corollary} \label{cor:morita-infinite}
    \(k[\Delta G]\) cannot be Morita equivalent to a category with finitely many objects.
\end{corollary}

\begin{proof}
    Suppose \(k[\Delta G]\) were Morita equivalent to a linear category with finitely many objects. Then \(\Mod_{k[\Delta G]}\) admits a finite tiny generator \(M_1, \dots, M_n\). Since each \(M_i\) is a retract of a finite direct sum of certain \(P_j\), there exists a finite direct sum \(S = P_{j_1} \oplus \dots \oplus P_{j_r}\) of which every \(M_i\) is a retract. If we choose \(n > 1 + \max \{ j_1, \dots, j_r \}\), then we must have \(\Mod_{k[\Delta G]}(S, P_n) = 0\) and thus \(\Mod_{k[\Delta G]}(M_i, P_n) = 0\) for all \(i\). However, since all \(P_n\) are nonzero, this contradicts the generating property of \((M_i)\).
\end{proof}

\begin{remark} \label{rem:crossed-structure-matrix}
    We can visualize \(\Nor(\Delta G)\) as an infinite matrix whose entries are the \(k\)-modules \(\Nor(\Delta G)(m,n)\). 
    If we let \(d = \dim G\), then the matrix has the following shape.
    \[
    \begin{tikzpicture}[x=0.75cm,y=0.75cm,font=\small]
    \begin{scope}
    \clip (0,-7) rectangle (7,0);

    \fill[black!10]
      (0,0) -- (1,0) -- (7,-6) -- (7,-7)
      -- (5,-7) -- (0,-2) -- cycle;

    \draw[
      black!45,
      line width=7pt,
      line cap=butt
    ] (0,0) -- (7,-7);

    \draw[black!65,semithick] (1,0) -- (7,-6);
    \draw[black!65,semithick] (0,-2) -- (5,-7);
  \end{scope}

  \draw[thick]
  (0.3,0) -- (0,0) -- (0,-7) -- (0.3,-7);

    \draw[thick]
  (6.7,0) -- (7,0) -- (7,-7) -- (6.7,-7);

  \node[above] at (3.5,0.15) {column \(n\)};
  \node[rotate=90] at (-0.75,-3.5) {row \(m\)};

  \node at (5.6,-1.2) {\Large \(0\)};
  \node at (1.2,-5.8) {\Large \(0\)};

  \node[
    rotate=-45,
    inner sep=0pt
  ] at (4.4,-5.3)
    {\(-1\leq m-n\leq d\)};

  \draw[stealth-stealth,semithick]
    (2,-1.08) -- (2,-3.92)
    node[
      midway,
      anchor=east,
      xshift=-3pt,
      yshift=2pt,
      inner sep=0pt
    ] {\(d+2\)};
    \end{tikzpicture}
    \]
    The entries above the diagonal are controlled by the \(\Delta\)-part of \(\Delta G\); they correspond to the linearized orbit of \(\delta_n\) under the right action of \(G_{n-1}\).
    The rest is controlled by the \(G\)-part of \(\Delta G\).
    Any nonzero entries in the leftmost column are transported diagonally.
    In particular, if \(G_r^{\nd} \neq \varnothing\) for infinitely many degrees \(r\), then every column has infinitely many nonzero entries below the diagonal. 
\end{remark}

\subsection{Classification of self-dual crossed simplicial groups}
A particularly nice situation arises when \(\Nor(\Delta G)\) is self-dual. 
We will show that in this case the structure matrix in \Cref{rem:crossed-structure-matrix} is symmetric.

\begin{lemma} \label{lem:crossed-simplicial-strong-equiv}
    Every equivalence \((\Delta G)^\op \simeq \Delta G\) is the identity on objects.
\end{lemma}

\begin{proof}
    We construct the following invariants.
    By a chain of morphisms we mean a finite sequence \((f_1, \dots, f_r)\) of composable morphisms.
    The chain starts at the domain of \(f_1\) and ends at the codomain of \(f_r\).
    Its length is the tuple size \(r\).
    \begin{itemize}
        \item Let \(h(n)\) be the supremum over all lengths of chains consisting of non-monic split epimorphisms starting at \(n \in \Delta G\).
        \item Let \(d(n)\) be the supremum over all lengths of chains consisting of non-epic split monomorphisms ending at \(n \in \Delta G\).
    \end{itemize}
    Clearly, we have \(h(n) = n = d(n)\).
    Now let \(F\colon (\Delta G)^\op \to \Delta G\) be an equivalence. Since \(F\) bijects the defining chains of \(h(n)\) to the defining chains of \(d(F(n))\) we conclude \(F(n) = d(F(n)) = h(n) = n\).
\end{proof}

\begin{corollary} \label{cor:self-dual-dim}
    If \((\Delta G)^\op \simeq \Delta G\), then \(\dim G = 1\).
\end{corollary}

\begin{proof}
    Let \(F\colon (\Delta G)^\op \to \Delta G\) be an equivalence.
    By the previous lemma, it is the identity on objects, and thus an isomorphism.
    We will show that every morphism \(f \in \Delta G(n, 0)\) for \(n \geq 2\) is degenerate.
    Write \(F(\sigma_j) = \delta_{i(j)} \circ g_j\) and \(F(f) = \phi \circ g\) where \(\phi \in \Delta(0, n)\).
    Note that \(j \mapsto i(j)\) is injective.
    Indeed, if \(i(j) = i(\ell)\) then
    \[ F(\sigma_j) = F(\sigma_\ell) \circ g_\ell^{-1} \circ g_j. \]
    Applying \(F^{-1}\) gives \(\sigma_j = h \circ \sigma_\ell\) for an automorphism \(h\) and \Cref{prop:crossed-reverse-codeg} implies \(j = \ell\).
    Choosing \(j\) such that \(i(j) \neq \phi(0)\) yields \(\phi = \delta_{i(j)} \circ \phi'\).
    We conclude
    \[ f = F^{-1}(\delta_{i(j)} \circ g_j \circ g_j^{-1} \circ \phi' \circ g) = F^{-1}(g_j^{-1} \circ \phi' \circ g) \circ \sigma_j \]
    and thus \(f\) is degenerate.
    For \(n = 1\) there is only \(F(\sigma_0) = \delta_{i(0)} \circ g_0\).
    The morphisms \(\phi \circ g \in \Delta G(0,n)\) which satisfy \(\phi(0) = i(0)\) give nondegenerate \(F^{-1}(\phi \circ g)\).
    They are indexed by any \(g \in G_0\), so \(G^\nd_1 \iso G_0 \neq \varnothing\) and \(\dim G = 1\).
\end{proof}

\begin{lemma} \label{lem:normalization-skeletal}
    \(\Nor(\Delta G)\) is skeletal.
\end{lemma}

\begin{proof}
    If \(m \iso n\), then we have \(\Nor(\Delta G)(m, x) \neq 0 \Leftrightarrow \Nor(\Delta G)(n, x) \neq 0\) for all \(x\).
    By \Cref{cor:crossed-normalization-nonzero}, the largest \(x\) with \(\Nor(\Delta G)(m, x) \neq 0\) is \(x = m + 1\). 
    We conclude \(m + 1 = n + 1\) and thus \(m = n\).
\end{proof}

\begin{theorem} \label{thm:crossed-self-strong-equiv}
     Every equivalence \(\Nor(\Delta G)^\op \simeq \Nor(\Delta G)\) is the identity on objects.
\end{theorem}

\begin{proof}
    Let \(F\colon \Nor(\Delta G)^\op \to \Nor(\Delta G)\) be an equivalence.
    Since \(\Nor(\Delta G)\) is skeletal, \(F\) is bijective on objects.
    Define \(r \in \bN\) by \(F(r) = F(0) + 1\). Then \(r > 0\) and
    \[ \Nor(\Delta G)(r, 0) \iso \Nor(\Delta G)(F(0), F(0) + 1) \neq 0. \]
    It follows that \(G_r^{\nd} \neq \varnothing\) and thus \(\Nor(\Delta G)(F(m), F(m+r)) \neq 0\) for all \(m\).
    We will now show by induction that \(F\) is the identity on objects on every initial segment of \(\bN\).
    For the empty segment this is trivial.
    Now suppose \(F(x) = x\) for \(x < m\).
    Then \(F\) restricts to a permutation on \(\bN \setminus [m-1]\).
    Observe that
    \begin{align*}
        \{ x \geq m : \Nor(\Delta G)(m, x) \neq 0 \} &= \{ m, m + 1 \} \\
        \iff \{ y \geq m : \Nor(\Delta G)(y, F(m)) \neq 0 \} &= \{ F(m), F(m+1) \}.
    \end{align*}
    Hence, there are two objects in \(\bN \setminus [m-1]\) with nonzero morphisms to \(F(m)\).
    Now, if \(F(m) > m\), then \(F(m) - 1\), \(F(m)\) and \(F(m) + r\) all belong to \(\bN \setminus [m-1]\) and admit nonzero morphisms to \(F(m)\), giving a contradiction.
    So we must have \(F(m) \leq m\) and thus \(F(m) = m\).
\end{proof}

\begin{corollary} \label{cor:self-dual-normal-dim}
    If \(\Nor(\Delta G)^\op \simeq \Nor(\Delta G)\), then \(\dim G = 1\).
\end{corollary}
\begin{proof}
    \Cref{thm:crossed-self-strong-equiv} implies \(\Nor(\Delta G)(m, 0) = 0 \Leftrightarrow \Nor(\Delta G)(0, m) = 0\).
    The latter holds for all \(m \geq 2\) by the rank formula.
    Now the result follows from \(\Nor(\Delta G)(m,0) \iso k[G_m^{\nd}]\) and \(\Nor(\Delta G)(0,1) \neq 0\).
\end{proof}

We can now classify the crossed simplicial groups with self-dual normalization.

\begin{theorem} \label{thm:dim-one-classification}
    Let \(\Delta G\) be a crossed simplicial group with associated extension
    \[ \Delta K \to \Delta G \xrightarrow{p} \Delta H \]
    as in \Cref{thm:classification-extensions} where \(\Delta K\) is the kernel of \(p\).
    The following are equivalent.
    \begin{enumerate}
        \item There exists an equivalence \(\Nor(\Delta G)^\op \simeq \Nor(\Delta G)\).
        \item The underlying simplicial set \(G\) has dimension one.
        \item \(\Delta H\) is the cyclic or dihedral category and \(K\) is constant.
        \item There exists an equivalence \((\Delta G)^\op \simeq \Delta G\).
    \end{enumerate}
\end{theorem}

\begin{proof}
    We begin with \((1) \Rightarrow (2) \Rightarrow (3) \Rightarrow (1)\) and add the last claim at the end.

    \begin{itemize}
        \item \((1) \Rightarrow (2)\): This is \Cref{cor:self-dual-normal-dim}.
        \item \((2) \Rightarrow (3)\): Suppose \(\dim G = 1\). 
            Since \(K\) is a simplicial subset of \(G\) and \(G \to H\) is degreewise surjective, we have \(\dim K \leq 1\) and \(\dim H \leq 1\).
            We note that \(\dim K = 0\) since \(K\) is a finite-dimensional Kan complex (cf.~\cite[Lemma~8.2.8]{weibel94}) and therefore, must be constant.
            Next we show \(\dim H = 1\).
            Suppose we had instead \(\dim H = 0\). Let \(p\colon G \to H\) be the morphism in the exact sequence.
            Let \(g \in G_1\) and write \(p(g) = s_0(q)\) for some \(q \in H_0\). 
            Since \(p\) is degreewise surjective, we may choose \(a \in G_0\) with \(p(a) = q\).
            Using exactness and \(K_1 = s_0(K_0)\) we conclude \(g = s_0(h) s_0(a) = s_0(ha)\) for some \(h \in K_0\). 
            However, this contradicts \(\dim G = 1\).
            Thus, we must have \(\dim H = 1\).
            Applying the EZ-decomposition to \(H\) yields \(|H_n| = |H_0| + n \cdot |H_1^{\nd}|\).
            Using the classification of \(\Delta H\) by growth rates \cite[Table 2]{dyckerhoff-kapranov-15}, we deduce that \(\Delta H\) is the cyclic or dihedral category.
        \item \((3) \Rightarrow (1)\): This argument uses the technical machinery we develop in \Cref{sec:lifting-through-extensions}.
                By \Cref{thm:extension-quotient} and \Cref{cor:induced-self-duality} we recover \(\Delta G\) as a quotient of the construction \(\Delta P\) and obtain a self-duality on \(\Delta G\) whose linearization bijects the families of Dold--Puppe projectors \(p_n \mapsto q_n\). 
                This gives an isomorphism \(\Nor(\Delta G)^\op \iso \Nor(\Delta G)\) as described in \Cref{rem:normalization-base-change}.
        \item \((4) \Rightarrow (2)\): This is \Cref{cor:self-dual-dim}.
        \item \((3) \Rightarrow (4)\): This is \Cref{cor:induced-self-duality}. \qedhere
    \end{itemize}
 
\end{proof}

\section{Normalization of the cyclic family} \label{sec:cyclic-family}

In this section, we apply Morita normalization to the cyclic family of planar crossed simplicial groups.
Our starting point will be the Dwyer--Kan correspondence, which provides the normalization of \(\Lambda_+\).
We adapt the proof strategy of Dwyer and Kan given in \cite{dwyer-kan85} to our framework.
Note that a more direct Dold--Kan correspondence for the paracyclic category is given by Getzler in \cite{getzler26}.
We adopt Getzler's notation of the Dwyer--Kan operator and the Karoubi operator and compare the approaches in \Cref{rem:getzler-comparison}.

\begin{convention}
    In this section we will use \(\Delta_{\min}\) instead of \(\Delta_{\max}\). 
    Translating via the order reversal automorphism means we use the projectors \(q_n\) rather than \(p_n\).
\end{convention}

\begin{definition}
    We denote by \((-)^\circ \colon \Lambda_{\infty}^\op \to \Lambda_{\infty}\) the self-duality in \Cref{prop:paracyclic-selfduality} which restricts to a self-duality \(\Lambda_+^\op \to \Lambda_+\)
    and define the \emph{extra codegeneracies}
    \[ \sigma^n_n \coloneq (\delta^n_0)^\circ = \sigma^n_0 \circ \tau_n \in \Lambda_+(n, n-1) \quad \quad (n \geq 1).
    \]
    Furthermore, we define \(\Delta^\circ\) as the image of \(\Delta\) under \((-)^\circ\).
\end{definition}

\begin{remark}
    Using the extra codegeneracies, the self-duality \((-)^\circ\) acts by
    \begin{align*}
        (\delta^n_i)^\circ = \sigma^n_{n-i}, \quad (\sigma^n_j)^\circ = \delta^n_{n-j}, \quad (\tau_n)^\circ = \tau_n
    \end{align*}
    for \(0 \leq i \leq n\) and \(0 \leq j < n\).
\end{remark}

The following lemma is the equivalent of \cite[\nopp Remark~3.6]{dwyer-kan85}.
\begin{lemma}
    The diagram of inclusions
    \[\begin{tikzcd}[ampersand replacement=\&,cramped]
        {\Delta_{\min}} \&\& \Delta \\
        \\
        {\Delta^\circ} \&\& {\Lambda_+}
        \arrow[hook, from=1-1, to=1-3]
        \arrow[hook, from=1-1, to=3-1]
        \arrow[hook, from=1-3, to=3-3]
        \arrow[hook, from=3-1, to=3-3]
    \end{tikzcd}\]
    is a pushout in \(\Delta_{\min} \downarrow \Cat\).
\end{lemma}

\begin{proof}
    Let \(\cP = \Delta^\circ \coprod_{\Delta_{\min}} \Delta\) be the pushout in \(\Cat\) and \(F\colon \cP \to \Lambda_+\) the universal morphism of cocones.
    Since \(F\) is necessarily bijective on objects, it suffices to show that it is fully faithful.
    However, this follows by the normal form of \(\Lambda_{\infty}\) (cf.~\Cref{cor:paracyclic-normalform}).
    In particular, we have a pushout square in \(\Delta_{\min} \downarrow \Cat\).
\end{proof}

\begin{definition}
    Let \(\Du\) be the free linear category on the arrows \((u_n \colon n \to n+1)_{n \in \bN}\) and \((d_n\colon n+1 \to n)_{n \in \bN}\) subject to the relations \(u_{n+1} \circ u_{n} = 0\) and \(d_n \circ d_{n+1} = 0\).
\end{definition}

\begin{theorem}[Dwyer--Kan correspondence]
    The functor \(F\colon \Du \to \Nor(\Lambda_+)\), induced by
    \[ n \mapsto n, \quad u_n \mapsto [\delta^{n+1}_0], \quad d_n \mapsto [(-1)^{n+1} \sigma_{n+1}^{n+1}] \]
    is an isomorphism.
    In particular, the abstract Dold--Kan correspondence gives an equivalence
    \[ \Fun(\Lambda_+^\op, k\Mod) \iso \Mod_{k[\Lambda_+]} \simeq \Mod_{\Nor(\Lambda_+)} \iso \Mod_{\Du}. \]
\end{theorem}

\begin{proof}
    Since \(\Nor(-)\) is an equivalence, the diagram of inclusions
    \[\begin{tikzcd}[ampersand replacement=\&,cramped]
        {\Nor(\Delta_{\min})} \&\& {\Nor(\Delta)} \\
        \\
        {\Nor(\Delta^\circ)} \&\& {\Nor(\Lambda_+)}
        \arrow[hook, from=1-1, to=1-3]
        \arrow[hook, from=1-1, to=3-1]
        \arrow[hook, from=1-3, to=3-3]
        \arrow[hook, from=3-1, to=3-3]
    \end{tikzcd}\]
    is a pushout in \(\Ring_{\Nor(\Delta_{\min})}\) and thus in \(k\Cat\).
    By \Cref{ex:trivial}, \(\Nor(\Delta_{\min})\) is the linearization of the discrete category \(\bN\). Translating the basis of \Cref{ex:classical-dk} to \(\Delta_{\min}\), we obtain \(\Nor(\Delta) \iso \Ch^{\geq 0}\) with basis morphisms \([\delta^n_0]\). 
    The self-duality \((-)^\circ\) of \(\Lambda_+\) yields \(\Nor(\Delta^\circ) \iso \Ch_{\geq 0}\) with basis morphisms \([(-1)^n \sigma^n_n]\).
    We recognize \(F\) as the isomorphism of pushouts \(\Du \iso \Ch_{\geq 0} \coprod_{k[\bN]} \Ch^{\geq 0} \to \Nor(\Lambda_+)\).
\end{proof}

\begin{remark}
    The remaining normalizations of the cyclic family can be obtained from \(\Lambda_+\).
    Recall that every morphism \(f \in \Lambda_+(m,n)\) satisfies \(f(x + m + 1) = f(x) + n + 1\) for \(x \in \bZ\).
    We obtain a natural transformation \(T\colon \id_{\Lambda_+} \Rightarrow \id_{\Lambda_+}\) with components \(T_n = \tau_n^{n+1}\).
    \begin{itemize}
        \item \(\Lambda_{\infty}\) is the localization of \(\Lambda_+\) at \(T\).
        \item \(\Lambda\) is the quotient of \(\Lambda_+\) by the relation \(T = \id\).
        \item \(\Lambda_N\) is the quotient of \(\Lambda_+\) by the relation \(T^N = \id\). 
    \end{itemize}
    All constructions mentioned above are preserved by \(\Nor(-)\) as we have seen in \Cref{sec:normalization-standard-constructions}.
\end{remark}

In view of the Dwyer--Kan correspondence, it is useful to express these constructions as localizations and quotients of \(\Du\) with respect to the natural transformation 
\[ \pi \coloneq F^{-1}(\mathsf{T})\colon \id_{\Du} \Rightarrow \id_{\Du}. \]
The following two insights let us compute the components of \(\pi\).
We begin by studying the center of \(\Du\), following \cite[\nopp Theorem~6.3]{dwyer-kan85}.

\begin{remark} \label{rem:center-du}
    Let \(x_n \coloneq d_n \circ u_n\) and \(y_n \coloneq u_{n-1} \circ d_{n-1}\) where \(y_0 \coloneq 0\).
    Essentially by construction, we have 
    \[ \Du(0,0) \iso k[x_0], \quad \text{and} \quad \Du(n,n) \iso k[x_n, y_n] / (x_n \circ y_n) \quad (n \geq 1). \]
    Let \(\alpha\colon \id_{\Du} \Rightarrow \id_{\Du}\) and write \(\alpha_n = c_n + f_n(x_n) + g_n(y_n)\) for \(f_n, g_n \in t \cdot k[t]\).
    Naturality of \(\alpha\) can be verified on \(u_n\) and \(d_n\); it translates to \(c_n = c_{n+1}\) and \(f_n = g_{n+1}\).
    In particular, \(\alpha\) is fully determined by the constant \(c_0\) and the polynomials \(f_n\).
\end{remark}

We can therefore recover \(\pi\) by determining the components \(\pi_n \bmod y_n\).
Taking the quotient \(\Du / \la \id_r : r < n \ra\) forces the downward loop \(y_n\) to become zero.
The resulting quotient functor takes \(\pi_n\) to \((\pi_n \bmod y_n)\).
As we will show now, this construction is also compatible with normalization.

\begin{proposition}
    We have \(\Nor(\cC / \la \id_r : r < n \ra) \iso \Nor(\cC) / \la \id_r : r < n \ra\) for every \(\cC \in \Ring_{k[\Delta_{\min}]}\). 
    In particular, normalization induces a \(k\)-algebra isomorphism
    \[ \cC(n,n) / I(n,n) \iso \Nor(\cC) / J(n,n) \]
    where \(I \subseteq \cC\) and \(J \subseteq \Nor(\cC)\) are the ideals \(\la \id_r : r < n \ra\) above.
\end{proposition}

\begin{proof}
    Let \(\cC\) be a \(k[\Delta_{\min}]\)-ring.
    Recall (cf.~\Cref{rem:projector-expansion}) that \(q_n - \id_n = \sum_{i=1}^n \delta_i \circ a_i'\) factors through objects \(r < n\).
    By induction, we conclude \(\la \id_r : r < n \ra = \la q_r : r < n \ra\).
    Therefore, its normalization is the ideal \(\la \id_r : r < n \ra \subseteq \Nor(\cC)\) which gives the isomorphism of categories.
    Since \([q_n]\) is the identity in \(\cC / I\), normalization is multiplicative at \(n\).
\end{proof}

We recover \cite[\nopp Proposition~6.5]{dwyer-kan85} directly from the index categories.

\begin{corollary} \label{cor:dwyer-kan-operator}
    \(\pi_n = (\id - x_n)^{n+1} + (\id - y_n)^{n} - \id = (\id - x_n)^{n+1} \circ (\id - y_n)^{n}\)
\end{corollary}

\begin{proof}
    The isomorphisms \(\Nor(k[\Lambda_+] / I) \iso \Nor(\Lambda_+) / I \iso \Du / I\) with \(I = \la \id_r : r < n \ra\) induce a \(k\)-algebra map
    \[ \phi\colon k[\Lambda_+](n,n) \to k[\Lambda_+](n,n) / I(n,n) \iso \Du(n,n) / I(n,n) \iso k[x_n] \]
    where \(\phi(f) = F^{-1}([f]) \bmod y_n\).
    In particular, we have \(\phi(\tau_n^{n+1}) = \phi(\tau_n)^{n+1}\).
    This reduces the problem to finding \(F^{-1}([\tau_n])\).
    Noting that \(\tau_n = \sigma_{n+1} \circ \delta_0\), we compute its relative \([\sigma_{n+1}] \circ [\delta_{0}]\) in \(\Nor(\Lambda_+)\).
    Using \(q_{n+1} \circ \delta_0 = \Rev(p_{n+1} \circ \delta_{n+1})\) and \Cref{ex:classical-dk}, we obtain \(q_{n+1} \circ \delta_0 = \big( \sum_{i=0}^{n+1} (-1)^i \delta_i \big) \circ q_{n}\).
    Therefore,
    \begin{align*}
        \sigma_{n+1} \circ q_{n+1} \circ \delta_0 &= \big( \sigma_{n+1} \circ \delta_0 + (-1)^{n+1} \sigma_n \circ \delta_{n+1} \big) \circ q_n + \sum_{i=1}^{n} (-1)^i \delta_i \circ c_i \\
        &= \big( \tau_n + (-1)^{n+1} \, \id_n \big) \circ q_n + \psi
    \end{align*}
    for certain \(c_i\). In particular, we have \([\psi] = 0\) and thus \([\sigma_{n+1}] \circ [\delta_0] = [\tau_n] + (-1)^{n+1} \, \id\).
    Using \(F(u_n) = [\delta_0]\) and \(F(d_n) = (-1)^{n+1} [\sigma_{n+1}]\), we conclude \(F^{-1}([\tau_n]) = (-1)^n (\id - x_n)\) and thus, \(\pi_n \bmod y_n = \big( (-1)^n (\id - x_n) \big)^{n+1} = (\id - x_n)^{n+1}\). Now the components of \(\pi\) can be reconstructed from \Cref{rem:center-du} via
    \[ c_n = 1, \quad f_n(t) = (1 - t)^{n+1} - 1, \quad g_n(t) = (1 - t)^n - 1. \qedhere \]
\end{proof}

\begin{remark} \label{rem:getzler-comparison}
    In \cite[\nopp Theorem~4.2]{getzler26} Getzler characterizes \(\Mod_{\Nor(\Lambda_{\infty})}\) as duchain complexes with invertible Karoubi operator.
    We can recover a similar result from Morita normalization.
    To obtain matching sign conventions, we define the Karoubi operator \(\kappa\colon \id_{\Du} \Rightarrow \id_{\Du}\) with components
    \[ \kappa_n \coloneq \id - (x_n + y_n) = (\id - x_n) \circ (\id - y_n). \]
    Under \(\Du \iso \Nor(\Lambda_+)\) it corresponds to \(F(\kappa_n) = (-1)^n [\tau_n - \delta_0 \circ \sigma_n]\) and \(F(\kappa_0) = [\tau_0]\).
    Using \Cref{cor:dwyer-kan-operator}, we obtain the identities \(\pi_n = \kappa_n^n \circ (\id - x_n)\) and \(\kappa_n^{n+1} = \pi_n \circ (\id - y_n)\).
    It follows that \(\pi_n\) is invertible if and only if \(\kappa_n\) is invertible.
    This gives an isomorphism of linear categories \(\Nor(\Lambda_\infty) \iso \Du[\pi^{-1}] \iso \Du[\kappa^{-1}]\).
\end{remark}

\begin{corollary} \label{cor:cyclic-family-normalized}
    We summarize the normalization of the cyclic family.
    \begin{align*}
        \Nor(\Lambda_+) &\iso \Du, \\
        \Nor(\Lambda_{\infty}) &\iso \Du[\pi^{-1}] \iso \Du[\kappa^{-1}], \\
        \Nor(\Lambda) &\iso \Du / \la \pi - \id \ra, \\
        \Nor(\Lambda_{N}) &\iso \Du / \la \pi^N - \id \ra. 
    \end{align*}
\end{corollary}

\section{Self-duality of crossed simplicial groups} \label{sec:crossed-self-duality}

Following the work of Elmendorf \cite{elmendorf93}, we are going to construct the paracyclic category by suitably extending \(\Delta\) in such a way that order adjoints always exist.
We recall (cf.~\Cref{rem:order-adjoint-chain}) that we have a chain of adjunctions \(\delta^n_{i+1} \dashv \sigma^n_i \dashv \delta_i^n\) for \(0 \leq i < n\). 
However, \(\delta^n_n\) lacks a left adjoint and \(\delta^n_0\) lacks a right adjoint. 
We can remedy this problem by extending \(\phi\colon [m] \to [n]\) to \(\tilde \phi\colon \bZ \to \bZ\) via the rule \(\tilde{\phi}(x + m + 1) = \tilde{\phi}(x) + n + 1\). 
This idea leads to the definition of the paracyclic category which we introduce below.

\begin{proposition}
    Let \(f\colon \bZ \to \bZ\) be weakly increasing. Then \(f\) admits a left adjoint if and only if \(f\) is unbounded in both directions. In this case, the left adjoint \(Lf \dashv f\) can be computed via \(Lf(x) = \min \{ y \in \bZ : x \leq f(y) \}\).
\end{proposition}

\begin{proof}
    The adjunction condition is \(Lf(x) \leq y \iff x \leq f(y)\) for all \(x,y \in \bZ\). Consider the sets \(S(x) \coloneq \{ y \in \bZ : x \leq f(y) \}\). Assuming \(Lf\) exists, the adjunction implies \(Lf(x) = \min S(x)\). Note that \(S(x)\) is nonempty for all \(x\) precisely if \(f\) admits no upper bound. Furthermore, every \(S(x)\) admits a minimum precisely if \(f\) admits no lower bound. 
\end{proof}

\begin{corollary} \label{cor:existence-adjoints}
    Every weakly increasing map \(f\colon \bZ \to \bZ\) which is unbounded in both directions admits left and right adjoints \(Lf \dashv f \dashv Rf\). The right adjoint can be computed via \(Rf(y) = \max \{ x \in \bZ : f(x) \leq y \}\). The assignments \(f \mapsto Lf\) and \(f \mapsto Rf\) are mutual inverses. Furthermore, we have \(Rf(y) = Lf(y+1)-1\).
\end{corollary}

\begin{proof}
    Since the poset category \(\bZ\) is self-dual, a left adjoint of \(f\) exists if and only if a right adjoint exists. Hence, if \(f\) admits adjoints, so do \(Lf\) and \(Rf\). By uniqueness of adjoints, we have \(RLf = f = LRf\). The defining equation for \(Rf\) is obtained by formal dualization. It remains to show \(Rf(y) = Lf(y+1)-1\). For all \(x,y \in \bZ\) we have 
    \begin{align*}
        x \leq Rf(y) &\iff f(x) \leq y \\
        &\iff \neg(y + 1 \leq f(x)) \\
        &\iff \neg (Lf(y+1) \leq x) \\
        &\iff x + 1 \leq Lf(y+1) \\
        &\iff x \leq Lf(y + 1) - 1,
    \end{align*}
    hence the result follows by uniqueness of the adjoint.
\end{proof}

\begin{proposition} \label{prop:reverse-spiral}
    Let \(f\colon \bZ \to \bZ\) be weakly increasing. If \(f(x + m + 1) = f(x) + n + 1\) for all \(x \in \bZ\), then \(Lf(x + n + 1) = Lf(x) + m + 1\) for all \(x \in \bZ\). The same holds for \(Rf\).
\end{proposition}

\begin{proof}
    We apply the adjunction:
    \begin{align*}
        Lf(x + n + 1) \leq y &\iff x \leq f(y) - n - 1 = f(y - m - 1) \\
        &\iff Lf(x) + m + 1 \leq y.
    \end{align*}
    Since this holds for all \(x,y \in \bZ\), we conclude \(Lf(x + n + 1) = Lf(x) + m + 1\). The result for \(Rf\) follows by duality.
\end{proof}

\subsection{The cyclic family}

\begin{definition}
    The \emph{paracyclic category} \(\Lambda_{\infty}\) is given by the following data. The objects are natural numbers and a morphism \(f \in \Lambda_{\infty}(m,n)\) is given by a weakly increasing map \(f\colon \bZ \to \bZ\) satisfying the spiraling property \(f(x + m + 1) = f(x) + n + 1\) for all \(x \in \bZ\). Morphisms are composed as maps of sets. We identify \(\Delta\) as a wide subcategory of \(\Lambda_{\infty}\) by extending the morphisms of \(\Delta\) via the spiraling condition. Furthermore, we denote by \(\tau_n \in \Lambda_{\infty}(n,n)\) the automorphism given by the map \(\tau\colon \bZ \to \bZ\) where \(\tau(x) = x + 1\).
\end{definition}

\begin{theorem} \label{thm:paracyclic-self-duality}
    The assignment \(f \mapsto Lf\) extends to a self-duality \(L\colon \Lambda_{\infty}^\op \to \Lambda_{\infty}\) which is the identity on objects. Its inverse is given by \(R\).
\end{theorem}

\begin{proof}
    This follows immediately from \Cref{cor:existence-adjoints} and \Cref{prop:reverse-spiral}.
\end{proof}

\begin{remark}
    The self-duality \(L\) of \(\Lambda_{\infty}\) we have constructed is the one given by Elmendorf in \cite{elmendorf93}. By direct computation using the adjunctions, we obtain
    \[
        L(\delta^n_i) = \begin{cases}
            \sigma_i^n, & \text{if } i < n, \\
            \sigma_0^n \circ \tau_n, & \text{if } i = n
        \end{cases}, 
        \quad L(\sigma_j^n) = \delta_{j+1}^n, \quad L(\tau_n) = \tau_n^{-1}.
    \]
    In particular, we have \(L(p_n) = q_n = \Rev(p_n)\). While this is already enough to show that \(\Nor(\Lambda_{\infty})\) is self-dual, we present two modifications below that preserve the respective projector families and the shift automorphisms \(\tau_n\).
\end{remark}

\begin{definition} \label{def:twisting-automorphism}
    For every sequence of integers \(a = (a_n)_{n \in \bN} \subseteq \bZ\) we define a \emph{twisting automorphism} \(T_a\colon \Lambda_{\infty} \to \Lambda_{\infty}\) which is the identity on objects and maps \(f \in \Lambda_{\infty}(m,n)\) to \(T_a f\) where \(T_a f(x) \coloneq a_n - f(a_m - x)\) for all \(x \in \bZ\).
\end{definition}

\begin{definition}
    We denote by \(\Lambda_+\) the wide subcategory of \(\Lambda_{\infty}\) where 
    \[ \Lambda_{+}(m,n) = \{ f \in \Lambda_{\infty}(m,n) : f(0) \geq 0 \}. \]
\end{definition}

\begin{proposition} \label{prop:paracyclic-selfduality}
    Let \(a = (a_n)\) and \(b = (b_n)\) be the sequences \(a_n = n\) and \(b_n = 0\).
    \begin{enumerate}
        \item \(T_a \circ L\) preserves the family \((p_n)\).
        \item \(T_b \circ L \eqcolon (-)^\circ \) preserves the family \((q_n)\) and restricts to a self-duality of \(\Lambda_+\).
    \end{enumerate}
    Furthermore, both self-dualities preserve \(\tau_n\).
\end{proposition}

\begin{proof}
    \(T_a\) acts on \(\Delta \subseteq \Lambda_{\infty}\) as \(\Rev\), which shows (1). By direct computation, we see that \(T_b \circ L\) preserves \((q_n)\). Using the adjunction, we have \(T_b L f(0) \geq 0 \Leftrightarrow f(0) \geq 0\) which shows claim (2). To see that both self-dualities preserve \(\tau_n\), we note that in fact \(T_c(\tau_n) = \tau_n^{-1}\) for every sequence \(c = (c_n)\) of integers.
\end{proof}

\begin{remark}
    The functor \((-)^\circ\colon \Lambda_+^\op \to \Lambda_+\) has been used by Kowalzig and Krähmer
    \cite{kowalzig-kraehmer-2011} to construct generalized cyclic dualities.
\end{remark}

\begin{proposition}
    The automorphisms of the object \(n \in \Lambda_{\infty}\) are precisely \(\{ \tau_n^a : a \in \bZ \}\).
\end{proposition}

\begin{proof}
    Let \(f \in \Lambda_{\infty}(n,n)\) be an automorphism. The underlying map \(f\colon \bZ \to \bZ\) is a strictly increasing bijection, so it satisfies \(f(x + 1) = f(x) + 1\) for all \(x \in \bZ\). By induction, we conclude \(f(x) = f(0) + x\) and thus \(f = \tau_n^{a}\) where \(a = f(0)\).
\end{proof}

\begin{corollary} \label{cor:paracyclic-normalform}
    \(\Lambda_{\infty}\) is a crossed simplicial group with \(\Aut_{\Lambda_{\infty}}(n) \iso \bZ\).
\end{corollary}

\begin{proof}
    Let \(f \in \Lambda_{\infty}(m,n)\). \(Lf \in \Lambda_{\infty}(n,m)\) factors uniquely as \(Lf = \tau_m^a \circ \psi\) with \(a = Lf(0)\) and \(\psi(0) = 0\). By \Cref{thm:paracyclic-self-duality}, we have \(\psi = L \phi\) for a unique \(\phi \in \Delta(m,n)\). It follows that \(f\) has the unique factorization \(f = \phi \circ R(\tau_m)^a\). We immediately see from the adjunction that \(R(\tau_n) = \tau_n^{-1}\), so \(\Lambda_{\infty}\) is a crossed simplicial group.
\end{proof}

\begin{definition} \label{def:cyclic-category}
    For \(N \in \bN_{\geq 1}\) we define the \emph{\(N\)-cyclic category}
    \[ \Lambda_N \coloneq \Lambda_{\infty} / \la \tau_n^{N(n+1)} = \id_n : n \geq 0 \ra. \]
    The \emph{cyclic category} is defined as \(\Lambda \coloneq \Lambda_1\). 
\end{definition}

\begin{proposition}
    \(\Lambda_N\) is a crossed simplicial group where \(\Aut_{\Lambda_N}(n) \iso \bZ / N(n+1)\).
\end{proposition}

\begin{proof}
    It is readily verified that \(\la \tau_n^{N(n+1)} = \id_n : n \geq 0 \ra\) identifies \(f, f' \in \Lambda_{\infty}(m,n)\) precisely if 
    \[ f = f' \circ \tau_m^{\ell N(m+1)} \quad \text{ for some } \ell \in \bZ. \]
    Now the result follows from the normal form of \(\Lambda_{\infty}\). 
\end{proof}

\subsection{The dihedral family}

\begin{definition}
    We denote by \(\rho\colon \bZ \to \bZ\) the map \(\rho(x) = -x\).
    The \emph{paradihedral category} \(\Xi_{\infty}\) has objects \(n \in \bN\).
    Morphisms are given by 
    \[ \Xi_{\infty}(m,n) = \Lambda_{\infty}(m,n) \cup (\rho \circ \Lambda_{\infty}(m,n)) \] 
    with composition as maps of sets.
    We denote by \(\rho_n \in \Xi_{\infty}(n,n)\) the morphism corresponding to the map \(\rho\).
\end{definition}

\begin{proposition}
    \(\Xi_{\infty}\) is a crossed simplicial group with \(\Aut_{\Xi_{\infty}}(n) \iso D_{\infty}\).
\end{proposition}

\begin{proof}
    By definition, every morphism in \(\Xi_{\infty}(m,n)\) factors uniquely as \(\rho^{\eps}_n \circ f\) where \(\eps \in \{0, 1\}\) and \(f \in \Lambda_{\infty}(m,n)\).
    The automorphism \(T_b\colon \Lambda_{\infty} \to \Lambda_{\infty}\) in \Cref{prop:paracyclic-selfduality} maps \(T_b(f) = \rho_n \circ f \circ \rho_m\).
    Using the unique normal form \(T_b^{\eps}(f) = \phi \circ \tau_m^a\), we conclude 
    \[ \rho_n^\eps \circ f = T_b^{\eps}(f) \circ \rho_m^\eps = \phi \circ \tau_m^a \circ \rho^{\eps}_m \]
    for unique \(\phi \in \Delta(m,n)\), \(a \in \bZ\) and \(\eps \in \{0, 1\}\). In particular, this implies the claim \(\Aut_{\Xi_{\infty}}(n) \iso \bZ/2 \ltimes \bZ \iso D_{\infty}\).
\end{proof}

\begin{definition}
    For \(N \in \bN_{\geq 1}\) we define the \emph{\(N\)-dihedral category} as the quotient
    \[ \Xi_N \coloneq \Xi_{\infty} / \la \tau_n^{N(n+1)} = \id_n : n \geq 0 \ra. \]
    The \emph{dihedral category} is defined as \(\Xi \coloneq \Xi_1\).
\end{definition}

\begin{proposition}
    \(\Xi_N\) is a crossed simplicial group where \(\Aut_{\Xi_{N}}(n) \iso D_{N(n+1)}\).
\end{proposition}

\begin{proof}
    It is readily verified that \(\la \tau_n^{N(n+1)} = \id_n : n \geq 0 \ra\) identifies \(f,f' \in \Xi_{\infty}(m,n)\) precisely if
    \[ f = f' \circ \tau_m^{\ell N(m + 1)} \quad \text{ for some } \ell \in \bZ. \]
    Now the result follows from the normal form of \(\Xi_{\infty}\) and \(\rho_n \circ \tau_n = \tau_n^{-1} \circ \rho_n\).
\end{proof}

\begin{remark}
    The way we have constructed \(\Xi_{\infty}\) already suggests how we can extend the self-duality from \(\Lambda_{\infty}\).
    If \(\widehat{L}\colon \Xi_{\infty}^\op \to \Xi_{\infty}\) is an extension of \(L\), then for all \(f \in \Lambda_{\infty}(m,n)\) it must satisfy (in view of \Cref{cor:existence-adjoints})
    \begin{align*}
        Lf \circ \widehat L(\rho_n) &= \widehat L(\rho_n \circ f) = \widehat{L}(T_b(f) \circ \rho_m) = \widehat{L}(\rho_m) \circ L(T_b(f)) = \widehat{L}(\rho_m) \circ T_b(Rf) \\
        &= \widehat{L}(\rho_m) \circ \rho_m \circ \tau_m^{-1} \circ Lf \circ \tau_n \circ \rho_n.
    \end{align*}
    A natural choice for the extension is therefore \(\widehat{L}(\rho_n) = \tau_n \circ \rho_n\).
\end{remark}

\begin{proposition} \label{prop:order-adjoint-paradihedral}
    The self-duality \(L\colon \Lambda_{\infty}^\op \to \Lambda_{\infty}\) extends to a self-duality \(\widehat{L}\colon \Xi_{\infty}^\op \to \Xi_{\infty}\) which acts via
    \[
        \widehat L(\delta_i^n) = \begin{cases}
            \sigma_i^n & \text{if } i < n, \\
            \sigma_0^n \circ \tau_n & \text{if } i = n, 
        \end{cases}
        \quad
        \widehat{L}(\sigma_j^n) = \delta^n_{j+1},
        \quad
        \widehat L(\tau_n) = \tau_n^{-1}, \\
        \quad
        \widehat L(\rho_n) = \tau_n \circ \rho_n.
    \]
    Its inverse is the extension \(\widehat{R}\colon \Xi_{\infty}^\op \to \Xi_{\infty}\) of \(R\colon \Lambda_{\infty}^\op \to \Lambda_{\infty}\) by \(\widehat{R}(\rho_n) = \rho_n \circ \tau_n \).
\end{proposition}

\begin{proof}
    Since every morphism in \(\Xi_{\infty}(m,n)\) can be uniquely written in the form \(f \circ \rho_m^{\eps}\) with \(f \in \Lambda_{\infty}(m,n)\) and \(\eps \in \{0, 1\}\), we see that \(\widehat{L}\) is a well-defined map on morphisms.
    The only non-trivial case to verify for functoriality is (with suppressed indices)
    \begin{align*}
        \widehat{L}(f \circ \rho) \circ \widehat{L}(g \circ \rho) &= \tau \circ \rho \circ Lf \circ \tau \circ \rho \circ Lg \\
        &= \tau \circ T_b(Lf) \circ \tau^{-1} \circ Lg  \\
        &= \tau \circ R(T_b(f)) \circ \tau^{-1} \circ Lg \\
        &= L(T_b(f)) \circ Lg \\
        &= \widehat{L}(g \circ T_b(f)) \\
        &= \widehat{L}(g \circ \rho \circ f \circ \rho). 
    \end{align*}
    The argument for \(\widehat{R}\) is analogous.
    The actions of \(\widehat{L}\) on cofaces, codegeneracies and the shift are already given by \(L\), in which case we simply verify their adjunction directly.
\end{proof}

\begin{corollary} \label{cor:paradihedral-self-duality}
    Extending the morphisms in \Cref{def:twisting-automorphism} to \(\Xi_{\infty}\) and using \Cref{prop:paracyclic-selfduality} we obtain the following results.
    \begin{enumerate}
        \item \(T_a \circ \widehat{L}\) preserves the family \((p_n)\) and restricts to \(T_a \circ L\) on \(\Lambda_{\infty}\). 
        \item \(T_b \circ \widehat{L} \eqcolon (-)^\circ\) preserves the family \((q_n)\) and restricts to \(T_b \circ L\) on \(\Lambda_{\infty}\).
    \end{enumerate}
\end{corollary}

\subsection{Lifting self-duality to extensions} \label{sec:lifting-through-extensions}

Throughout this section we fix an extension \(p\colon \Delta G \to \Delta H\) of crossed simplicial groups with kernel \(\Delta K\) such that
\begin{itemize}
    \item \(\Delta H \in \{ \Lambda, \Xi \}\), and
    \item \(K\) is a constant simplicial group.
\end{itemize}
Since \(\Aut_{\Lambda}(0)\) is trivial, we regard it as a subgroup of \(\Aut_{\Xi}(0) \iso C_2 = \{1, -1\}\) and denote by \(\eps\colon G_0 \to C_2\) the composite group homomorphism \(G_0 \xrightarrow{p} H_0 \hookrightarrow C_2\).
We form the following pullback \(\Delta P\) in \(\Cat\).
\[\begin{tikzcd}[ampersand replacement=\&,cramped]
	{\Delta P} \&\& {\mathcal{B}G_0} \\
	\\
	{\Xi_{\infty}} \&\& {\mathcal{B}C_2}
	\arrow[from=1-1, to=1-3]
	\arrow[from=1-1, to=3-1]
	\arrow["\eps", from=1-3, to=3-3]
	\arrow["\sgn"', from=3-1, to=3-3]
\end{tikzcd}\]
Here, \(\mathcal{B}G\) is the delooping of a group into a one-object category and \(\sgn\) records whether a morphism in \(\Xi_{\infty}\) is increasing or decreasing.
We use the following explicit description of \(\Delta P\).
\begin{itemize}
    \item Objects of \(\Delta P\) are natural numbers \(n \in \bN\).
    \item Morphisms in \(\Delta P(m,n)\) are pairs \((f, a)\) where \(f \in \Xi_{\infty}(m,n)\), \(a \in G_0\). These are subject to the condition \(\sgn(f) = \eps(a)\).
    \item Composition is given by \((g, b) \circ (f, a) = (g \circ f, b \circ a)\).
\end{itemize}
Note that \(\Delta P\) is itself a crossed simplicial group and we can lift the order adjoint self-duality of \(\Xi_{\infty}\) to a self-duality on \(\Delta P\) via the functor
\[ (\Delta P)^\op \to \Delta P, \quad n \mapsto n, \quad (f, a) \mapsto ( \widehat{L}f, a^{-1} ). \]
Furthermore, \(\Delta P\) comes with a canonical morphism of crossed simplicial groups
\begin{align*}
    \pi\colon \Delta P \to \Delta H, \quad n \mapsto n, \quad (f, a) \mapsto [f].
\end{align*}
We will explicitly construct a lift:
\[\begin{tikzcd}[ampersand replacement=\&,cramped]
	\&\& {\Delta G} \\
	\\
	{\Delta P} \&\& {\Delta H}
	\arrow["p", from=1-3, to=3-3]
	\arrow[dashed, from=3-1, to=1-3]
	\arrow["\pi"', from=3-1, to=3-3]
\end{tikzcd}\]

\begin{definition}
    We denote by \(v_i^n \in \Delta(0, n)\) the \(i\)-th vertex operator with image \(i \in [n]\).
\end{definition}

\begin{convention}
    Since \(K\) is constant, all simplicial operators on it act as bijections.
    We identify the elements of \(K_0\) with elements in \(K_n\) via the action of the terminal morphism:
    \[ K_0 \xrightarrow{\sim} K_n, \quad k \mapsto k_n \coloneq k \triangleleft \omega_n. \]
    Note that \(k \mapsto k_n\) is an isomorphism of groups.
\end{convention}

\begin{lemma}
    Let \(f \in \Delta H(m,n)\).
    For every vertex operator \(v_i^m\) we have a bijection of fibers
    \[ p^{-1}(f) \to p^{-1}(f \circ v_i^m), \quad \overline{f} \mapsto \overline{f} \circ v^m_i. \]
\end{lemma}

\begin{proof}
    Write \(f = \phi \circ g\) in normal form and choose a lift \(\overline{g} \in G_m\).
    Then \(\overline{f} \coloneq \phi \circ \overline{g}\) is a lift of \(f\).
    By uniqueness of the normal form, the fiber over \(f\) must be indexed as
    \[ p^{-1}(f) = \{ \phi \circ \overline{g} \circ k_m : k \in K_0 \} = \{ \overline{f} \circ k_m : k \in K_0 \} \]
    Now write \(k_m \circ v^m_i = v^m_j \circ a\) with \(a \in G_0\). 
    Applying \(p\) yields \(v^m_i  = v^m_j \circ p(a)\).
    By uniqueness of the normal form we conclude \(i = j\) and \(a \in K_0\).
    In particular, \(a = k_m \triangleleft v_i^m\). Since \(K\) is constant, we must have \(a = k\).
    It follows that
    \[ k_m \circ v_i^m = v_i^m \circ k \]
    for all \(v_i^m\) and all \(k \in K_0\).
    Finally, note that \(\overline{f} \circ v_i^m\) is a lift of \(f \circ v_i^m\). Hence its fiber is indexed as
    \[ \{ \overline{f} \circ v_i^m \circ k : k \in K_0 \} = \{ \overline{f} \circ k_m \circ v_i^m : k \in K_0 \}. \qedhere \]
\end{proof}

Using the lemma, we see that lifts of morphisms are uniquely determined if we specify their value on a vertex.

\begin{construction}
    We construct a functor \(\Phi\colon \Delta P \to \Delta G\) such that \(p \circ \Phi = \pi\).
    For \(n \geq 1\) there exists a unique lift \(t_n \in G_n\) of the cyclic generator \(\tau_n \in H_n\) such that \(t_n \circ v_0^n = v_1^n\).
    Note that \(t_n \circ \delta_n = \delta_0\) because the same holds for the image under \(p\) and they match at vertex zero.
    In terms of vertex operators this implies
    \[ t_n \circ v^n_i = v^n_{i+1} \quad (0 \leq i < n). \]
    We define \(z \in K_0\) by the equation \(t_1 \circ v_1^1 = v_0^1 \circ z\).
    The last rotation is then given by 
    \[ t_n \circ v_n^n = v^n_0 \circ z. \]
    Indeed, let \(\phi, \psi \in \Delta(1, n)\) select vertices \((n-1, n)\) and \((0, n)\).
    Then \(\tau_n \circ \phi = \psi \circ \tau_1\).
    Their lifts \(t_n \circ \phi\) and \(\psi \circ t_1\) agree at vertex zero, so they are equal.
    Evaluating at the other vertex shows the claim.
    We set \(t_0 \coloneq z\).

    For \(a \in G_0\) set \(a_n \coloneq a \triangleleft \omega_n \in G_n\).
    Let \(r_{a,n}\) be the identity if \(\eps(a) = 1\) and the order reversal \(j \mapsto n - j\) otherwise.
    The normal forms and projection to \(\Delta H\) give
    \[ a_n \circ v_i^n = v^n_{r_{a,n}(i)} \circ a \quad \quad (0 \leq i \leq n). \]
    Indeed, composing with \(\omega_n\) gives the automorphism part as \(a\) and applying \(p\) gives the vertex index.
    Finally, the relation \(a \circ z = z^{\eps(a)} \circ a\) is readily verified at vertex zero.
    We can now extend the vertex operators \(v^n_j\) to all indices \(j \in \bZ\). Write \(j = q(n+1) + i\) for \(0 \leq i \leq n\) and set
    \[ v^n_j \coloneq v^n_i \circ z^q \in \Delta G(0, n). \]
    The previous equations involving vertex operators hold in the same way for \(j \in \bZ\).
    Moreover, identifying \(\phi \in \Delta(m,n)\) with its periodic extension to \(\bZ\) gives
    \[ \phi \circ v^m_j = v^n_{\phi(j)} \quad \quad (j \in \bZ). \]
    
    For \((f, a) \in \Delta P(m,n)\) let \(\Phi(f, a) \in \Delta G(m,n)\) be the unique morphism satisfying
    \[ p(\Phi(f,a)) = \pi(f,a), \quad \Phi(f,a) \circ v^m_0 = v^n_{f(0)} \circ a. \]
    In fact, the assignment satisfies the stronger property
    \[ \Phi(f, a) \circ v^m_j = v^n_{f(j)} \circ a \quad \quad (j \in \bZ). \]
    Indeed, let \(f = \phi \circ \tau_m^{\ell} \circ \rho_{a,m} \in \Xi_{\infty}(m,n)\) be the normal form. Then \(\pi(f, a)\) has a lift \(\phi \circ t_m^{\ell} \circ a_m\) in \(\Delta G\) and using the previous identities, we conclude
    \begin{align*}
        \phi \circ t_m^{\ell} \circ a_m \circ v_j^m &= \phi \circ t_m^{\ell} \circ v^m_{\rho_{a,m}(j)} \circ a \\
        &= v^n_{\phi(\rho_{a,m}(j) + \ell)} \circ a \\
        &= v^n_{f(j)} \circ a. 
    \end{align*}
    Functoriality of \(\Phi\) can now be verified directly at vertex \(0\).
\end{construction}

\begin{theorem} \label{thm:extension-quotient}
    There exists an element \(z \in K_0\) satisfying \(aza^{-1} = z^{\eps(a)}\) for all \(a \in G_0\) and an isomorphism of crossed simplicial groups \(\Delta P / \la (\tau_n^{n+1}, 1) = (\id_n, z) : n \geq 0 \ra \iso \Delta G\).
\end{theorem}

\begin{proof}
 We will show that \(\Phi\colon \Delta P \to \Delta G\) gives the required quotient presentation.
 First we show that \(\Phi\) is full.
 Let \(F \in \Delta G(m,n)\). Choose a representative \(f \in \Xi_{\infty}(m,n)\) of \(p(F) \in \Delta H(m,n)\) such that \(0 \leq f(0) \leq n\).
 Then we have
 \[ F \circ v^m_0 = v^n_{f(0)} \circ a \]
 for some \(a \in G_0\) with \(\eps(a) = \sgn(f)\), and thus \(F = \Phi(f, a)\).
 This shows that \(\Phi\) is full.

 It is left to characterize the kernel equivalence relation of \(\Phi\).
 We claim that
 \[ \Phi(f, a) = \Phi(g, b) \iff g = \tau_n^{(n+1)\ell} \circ f \quad \text{and} \quad b = z^{-\ell} \circ a \quad \text{for some } \ell \in \bZ.  \]
 Indeed, equality after applying \(p\) to the left side shows that \(f\) and \(g\) differ by a multiple of the full rotation \(\tau_n^{n+1}\).
 Their prescribed vertex values agree precisely when
 \[ v^n_{f(0)} \circ a = v^n_{g(0)} \circ b = v^n_{f(0)} \circ z^\ell \circ b \]
 which by uniqueness of this form, is equivalent to \(b = z^{-\ell} \circ a\).
 These identifications are precisely those generated by \((\tau_n^{n+1}, 1) \sim (\id_n, z)\) for \(n \geq 0\).
\end{proof}

\begin{corollary} \label{cor:induced-self-duality}
    The lifted order adjoint self-duality of \(\Delta P\) descends to a self-duality
    \[ (\Delta G)^\op \to \Delta G, \quad n \mapsto n, \quad \Phi(f, a) \mapsto \Phi(\widehat{L}(f), a^{-1}). \]
\end{corollary}

\begin{proof}
    The self-duality on \(\Delta P\) maps
    \begin{align*}
        (\tau_n^{n+1}, 1) \mapsto (\tau_n^{n+1}, 1)^{-1} \quad \text{and} \quad (\id_n, z) \mapsto (\id_n, z)^{-1} 
    \end{align*}
    so it preserves the generating relations.
    Its inverse does likewise, so it descends to a self-duality on \(\Delta G\). 
\end{proof}

\printbibliography

\end{document}